\documentclass[preprint,3p]{elsarticle}
\usepackage[toc,page,title,titletoc,header]{appendix}
\usepackage{stmaryrd}
\SetSymbolFont{stmry}{bold}{U}{stmry}{m}{n}
\usepackage{amsmath}
\usepackage{amssymb}
\usepackage{amsthm}
\usepackage{tabularx}
\usepackage{subfigure}
\usepackage{epsfig}
\usepackage{indentfirst}
\usepackage{cases}
\biboptions{sort&compress}
\usepackage{graphicx}
\usepackage{epsfig}
\usepackage{color}
\usepackage{float}
\usepackage{multirow}
\usepackage{mathtools}

\newcommand{\be}{\begin{equation}}
\newcommand{\ee}{\end{equation}}
\newcommand{\ba}{\begin{array}}
\newcommand{\ea}{\end{array}}
\newcommand{\bea}{\begin{eqnarray}}
\newcommand{\eea}{\end{eqnarray}}
\newcommand{\beas}{\begin{eqnarray*}}
\newcommand{\eeas}{\end{eqnarray*}}

\newtheorem{thm}{Theorem}[section]
\newtheorem{prop}{Proposition}[section]
\newtheorem{remark}{Remark}[section]

\newtheorem{lem}{Lemma}[section]
\newtheorem{defi}{Definition}[section]

\numberwithin{equation}{section}

\begin{document}

\begin{frontmatter}

\title{A stabilized parametric finite element method for surface diffusion \\ with an arbitrary surface energy}

\author[1]{Yulin Zhang}
\address[1]{School of Mathematical Sciences, MOE-LSC and Institute of Natural Sciences, Shanghai Jiao Tong University, Shanghai 200240, P. R. China}
\ead{yulin.zhang@sjtu.edu.cn}

\author[2]{Yifei Li\corref{4}}
\address[2]{Department of Mathematics, National
              University of Singapore, Singapore 119076}
\ead{liyifei@nus.edu.sg}
\cortext[4]{Corresponding author.}

\author[1]{Wenjun Ying}
\ead{wying@sjtu.edu.cn}


\begin{abstract}

We proposed a structure-preserving stabilized parametric finite element method (SPFEM) for the evolution of closed curves under anisotropic surface diffusion with an arbitrary surface energy $\hat{\gamma}(\theta)$. By introducing a non-negative stabilizing function $k(\theta)$ depending on $\hat{\gamma}(\theta)$, we obtained a novel stabilized conservative weak formulation for the anisotropic surface diffusion. A SPFEM is presented for the discretization of this weak formulation. We construct a comprehensive framework to analyze and prove the unconditional energy stability of the SPFEM under a very mild condition on $\hat{\gamma}(\theta)$. This method can be applied to simulate solid-state dewetting of thin films with arbitrary surface energies, which are characterized by anisotropic surface diffusion and contact line migration. Extensive numerical results are reported to demonstrate the efficiency, accuracy and structure-preserving properties of the proposed SPFEM with anisotropic surface energies $\hat{\gamma}(\theta)$ arising from different applications.

\end{abstract}



\begin{keyword}
Geometric flows, parametric finite element method, anisotropy surface energy, structure-preserving, area conservation, energy-stable
\end{keyword}

\end{frontmatter}


\section{Introduction}

Surface diffusion is a widespread process involving the movement of adatoms, molecules and atomic clusters at solid material interfaces \cite{oura2013surface}. Due to different surface lattice orientation, an anisotropic evolution process is generated for a solid material, which is called \textit{anisotropic surface diffusion} in the literature. Surface diffusion with an anisotropic surface energy plays an important role as a crucial mechanism and/or kinetics in various fields such as epitaxial growth \cite{fonseca2014shapes,gurtin2002interface}, surface phase formation \cite{ye2010mechanisms}, heterogeneous catalysis \cite{randolph2007controlling}, and other pertinent fields within surface and materials science \cite{shustorovich1991metal,barrett2007variational,shen2004direct}. In fact, broader and consequential applications of surface diffusion have been discovered in materials science and solid-state physics, notably in areas such as the crystal growth of nanomaterials \cite{gilmer1972simulation,gomer1990diffusion} and solid-state dewetting \cite{srolovitz1986capillary,wang2015sharp,jiang2020sharp,thompson2012solid, ye2010mechanisms,jiang2012phase,bao2017stable,jiang2018solid}.

\begin{figure}[t!]
\centering
\includegraphics[width=0.5\textwidth]{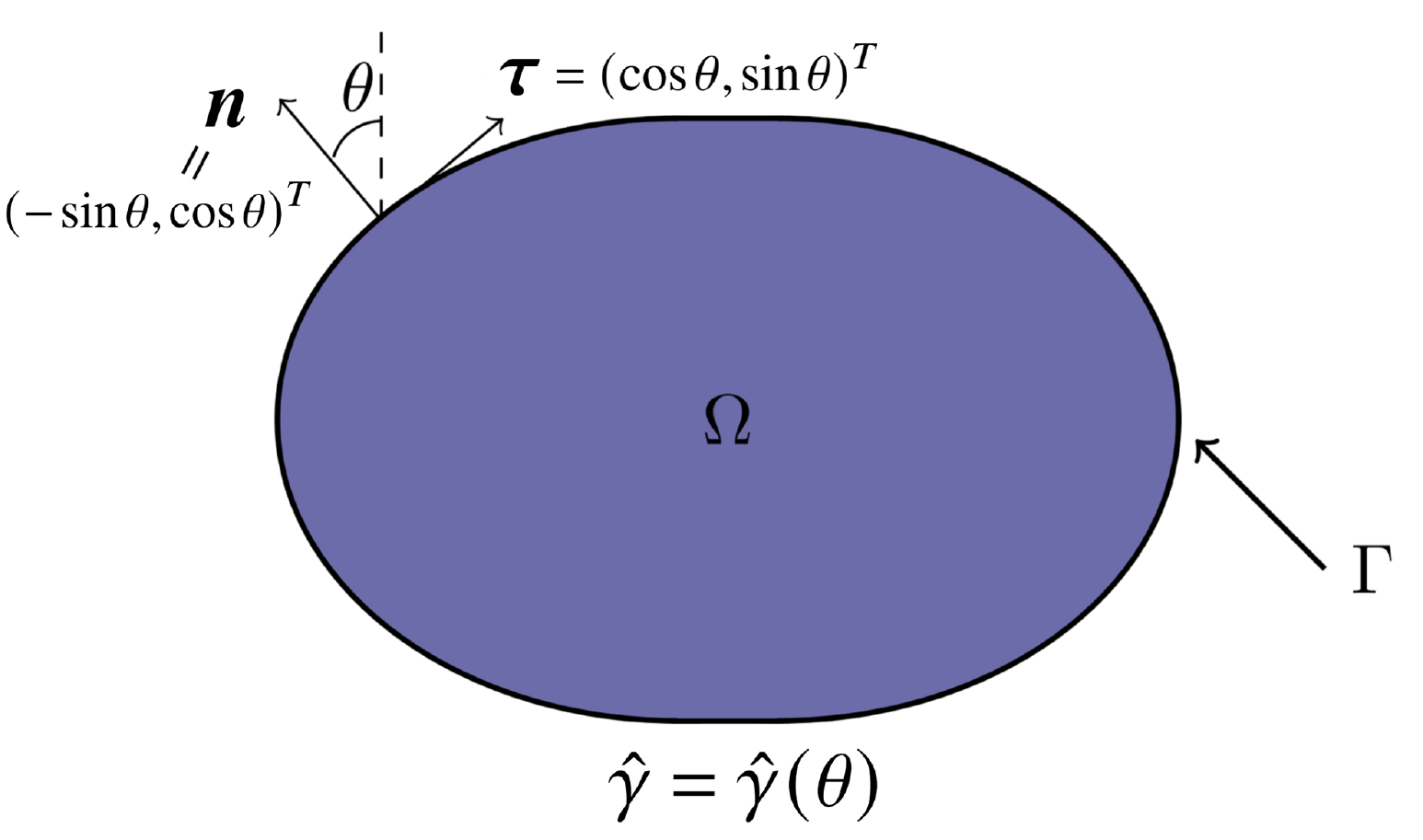}
\caption{An illustration of a closed curve under anisotropic surface diffusion with an anisotropic surface energy $\hat{\gamma}(\theta)$, while $\theta$ is the angle between the $y$-axis and the unit outward normal vector $\boldsymbol{n}=\boldsymbol{n}(\theta)\coloneqq(-\sin\theta,\cos\theta)^T$. $\boldsymbol{\tau}=\boldsymbol{\tau}(\theta)\coloneqq (\cos\theta,\sin\theta)^T$ represents the unit tangent vector.}
\label{fig: illustration of surface diffusion}
\end{figure}

 As shown in Fig.~\ref{fig: illustration of surface diffusion}, let $\Gamma\coloneqq\Gamma(t)$ be a closed curve in two dimensions (2D) associated with a given anisotropic surface energy $\hat{\gamma}(\theta)>0$, where $\theta\in2\pi\mathbb{T}\coloneqq\mathbb{R}/2\pi\mathbb{Z}$ represents the angle between the vertical axis and unit outward normal vector $\boldsymbol{n}=\boldsymbol{n}(\theta)\coloneqq(-\sin\theta,\cos\theta)^T$. It should be noted that the anisotropy  can also be viewed as a function $\gamma(\boldsymbol{n})$ of the normal vector $\boldsymbol{n}$ \cite{jiang2019sharp,jiang2020sharp,taylor1992ii}. While $\gamma(\boldsymbol{n}) = \gamma(-\sin\theta, \cos\theta)$ is equivalent to $\hat{\gamma}(\theta)$ by the one-to-one correspondence $\boldsymbol{n}(\theta)=(-\sin\theta, \cos\theta)^T$, the $\hat{\gamma}(\theta)$ formulation is often more convenient and straightforward in 2D.

 Suppose $\Gamma$ is represented by $\boldsymbol{X}\coloneqq\boldsymbol{X}(s,t)=(x(s,t),y(s,t))^T$, where $s$ denotes the arc-length parameter, and $t$ represents the time. The motion of $\Gamma$ under anisotropic surface diffusion is governed by the following geometric flow \cite{cahn1994overview,mullins1957theory}: \begin{equation}\label{eqn:surface diffusion}
    \partial_t\boldsymbol{X}=\partial_{ss}\mu\boldsymbol{n},
\end{equation} where $\mu$ is the weighted curvature (or chemical potential) defined as \begin{equation}\label{eqn:weighted curvature}
    \mu\coloneqq[\hat{\gamma}(\theta)+\hat{\gamma}^{\prime\prime}(\theta)]\kappa
\end{equation} with $\kappa\coloneqq-(\partial_{ss}\boldsymbol{X})\cdot\boldsymbol{n}$ being the curvature. 

The anisotropic surface diffusion \eqref{eqn:surface diffusion} is a fourth-order and highly nonlinear geometric flow, which possesses two major geometric properties, i.e., the area conservation and the energy dissipation. Let $A_c(t)$ be the area of the region $\Omega(t)$ enclosed by $\Gamma(t)$, and $W_c(t)$ be the total surface free energy, which are defined as \begin{equation}
    A_c(t)\coloneqq\int_{\Omega(t)}1\,\mathrm{d}x\mathrm{d}y=\int_{\Gamma(t)}y(s,t)\partial_sx(s,t)\,\mathrm{d}s,\qquad W_c(t)\coloneqq\int_{\Gamma(t)}\hat{\gamma}(\theta)\,\mathrm{d}s.
\end{equation} One can prove that \cite{bao2017parametric,bao2017stable,barrett2007variational} \begin{equation}
    \frac{\mathrm{d}}{\mathrm{d}t}A_c(t)=0,\qquad\frac{\mathrm{d}}{\mathrm{d}t}W_c(t)=-\int_{\Gamma(t)}|\partial_s\mu|^2\,\mathrm{d}s\leq 0,
\end{equation} which immediately implies the anisotropic surface diffusion \eqref{eqn:surface diffusion}--\eqref{eqn:weighted curvature} satisfies the area conservation and energy dissipation, i.e., \begin{equation}
    A_c(t)\equiv A_c(0),\qquad W_c(t)\leq W_c(t_1)\leq W_c(0),\qquad\forall t\geq t_1\geq0.
\end{equation}

When $\hat{\gamma}(\theta)\equiv 1,\forall\theta\in 2\pi\mathbb{T}$, the weighted curvature $\mu$ reduces to $\kappa$, and it is referred to as \textit{isotropic} surface energy. If $\hat{\gamma}(\theta)$ is not a constant function and $\hat{\gamma}(\theta)+\hat{\gamma}^{\prime\prime}(\theta)>0$ for all $\theta\in2\pi\mathbb{T}$, we classify the surface energy as \textit{weakly anisotropic}, otherwise, it is termed \textit{strongly anisotropic}. Typical anisotropic surface energies $\hat{\gamma}(\theta)$ which are widely employed in materials science include: \begin{enumerate}
    \item the $m$-fold anisotropic surface energy \cite{bao2017parametric} \begin{equation}\label{eqn:k-fold anisotropy}
        \hat{\gamma}(\theta)=1+\beta\cos m(\theta-\theta_0),\qquad\theta\in 2\pi\mathbb{T},
    \end{equation} where $m=2,3,4,6,|\beta|<1$ are dimensionless strength constants, $\theta_0\in 2\pi\mathbb{T}$ is a constant. Note that $\hat{\gamma}(\theta)$ is weakly anisotropic when $|\beta|<\frac{1}{m^2-1}$, and strongly anisotropic otherwise.
    \item the ellipsoidal anisotropic surface energy \cite{taylor1994linking} \begin{equation}\label{eqn:ellipsoidal anisotropy}
        \hat{\gamma}(\theta)=\sqrt{a+b\cos^2\theta},\qquad\theta\in2\pi\mathbb{T},
    \end{equation} where $a,b$ are two dimensionless constants satisfying $a>0$ and $a+b>0$.
        \item the Riemannian-like metric (also called BGN) anisotropic surface energy \cite{barrett2008numerical,barrett2008variational} \begin{equation}\label{eqn:BNG anisotropy}
        \hat{\gamma}(\theta)=\sum_{l=1}^L\sqrt{\boldsymbol{n}(\theta)^TG_l\boldsymbol{n}(\theta)},\qquad\boldsymbol{n}=(-\sin\theta,\cos\theta)^T,\qquad\theta\in 2\pi\mathbb{T},
    \end{equation} where $L\geq 1$, $G_l\in\mathbb{R}^{2\times 2}$ positive definite $\forall 1\leq l\leq L$. When $L=1,G_1=\text{diag}(a,a+b)$ in \eqref{eqn:BNG anisotropy}, the Riemannian-like metric anisotropy \eqref{eqn:BNG anisotropy} collapses to the ellipsoidal anisotropic surface energy \eqref{eqn:ellipsoidal anisotropy}.
    \item the piecewise anisotropic surface energy \cite{deckelnick2005computation} \begin{equation}\label{eqn:piecewise anisotropy}
        \hat{\gamma}(\theta)=\sqrt{\left(\frac{5}{2}+\frac{3}{2}\text{sgn}(n_1)\right)n_1^2+n_2^2},
    \end{equation} with $\boldsymbol{n}=(n_1,n_2)^T\coloneqq(-\sin\theta,\cos\theta)^T,\,\,\theta\in 2\pi\mathbb{T}.$
\end{enumerate}

Many numerical methods have been proposed for simulating the evolution of curves under surface diffusion, including the phase-field method \cite{tang2020, du2020phase,jiang2012phase,garcke2023diffuse}, the discontinuous Galerkin method \cite{xu2009local}, the marker particle method \cite{wong2000periodic, du2010tangent} and the parametric finite element method (PFEM) \cite{barrett2007parametric,barrett2008parametric,barrett2020parametric,bansch2005finite}. Among these methods, the energy-stable PFEM (ES-PFEM) by Barrett, Garcke, and N{\"u}rnberg \cite{barrett2007parametric}, also referred to as BGN's method, achieves the best performance in terms of mesh quality and unconditional energy-stability in the isotropic case. The ES-PFEM has been further extended to other geometric flows, such as the solid-state deweting problem \cite{bao2017parametric,jiang2016solid,wang2015sharp}, demonstrating its adaptability and robustness. Furthermore, Bao and Zhao have recently developed a structure-preserving PFEM (SP-PFEM) \cite{bao2021structure,bao2022volume,bao2023structure}, which can preserve the enclosed mass at the fully-discretized level while also maintaining the unconditional energy stability. Extending these PFEMs to anisotropic surface diffusion is a major focus of recent research in surface diffusion. While BGN successfully adapted their methods to a specific Riemannian metric form \cite{barrett2008numerical,barrett2008variational}, designing a SP-PFEM for anisotropic surface diffusion with arbitrary anisotropies  remains challenging.

Lately, based on the $\hat{\gamma}(\theta)$ formulation, Li and Bao introduced a surface energy matrix $G(\theta)$ and extended the ES-PFEM from the isotropic cases to the anisotropic cases \cite{li2021energy}. Due to the absence of a stabilizing term, their method requires a certain restrictive condition to ensure the energy stability. Subsequently, Bao, Jiang, and Li incorporated a stabilizing function within the $\gamma(\boldsymbol{n})$ formulation. They constructed a symmetric surface energy matrix $\boldsymbol{Z}_k(\boldsymbol{n})$ and proposed a symmetrized SP-PFEM for the anisotropic surface diffusion in \cite{bao2023symmetrized2D,bao2023symmetrized3D}. The symmetrized SP-PFEM with the stabilizing function works effectively for symmetric surface energy distributions (i.e., $\gamma(\boldsymbol{n})=\gamma(-\boldsymbol{n})$) to maintain the geometric properties. However, there are different anisotropic surface energies $\hat{\gamma}(\theta)$ which are not symmetrically distributed or do not satisfy the specific condition, such as the $3$-fold anisotropic surface energy \eqref{eqn:k-fold anisotropy} \cite{bao2017parametric,wang2015sharp}. Very recently, based on the $\gamma(\boldsymbol{n})$ formulation, Bao and Li introduced a novel surface energy matrix and established a comprehensive analytical framework to demonstrate the unconditional energy stability of the proposed SP-PFEM \cite{bao2024structure,bao2023unified}. Based on this framework, they successfully reduced the requirement for the anisotropy to $3\gamma(\boldsymbol{n})>\gamma(-\boldsymbol{n}),\,\,\forall\boldsymbol{n}\in\mathbb{S}^1$. 

However, the critical situation $3\gamma(\boldsymbol{n}^*)=\gamma(-\boldsymbol{n}^*),\,\,\boldsymbol{n}^*\in\mathbb{S}^1$ was not addressed in their analytical framework. This is because their framework relies on an estimate of the gradient of $\gamma(\boldsymbol{n})$. The challenge comes from $\gamma(\boldsymbol{n})$ being a function defined on the unit sphere $\mathbb{S}^1$, which makes its gradient complicated to analyze. In contrast, the $\hat{\gamma}(\theta)$ formulation, defined on $2\pi\mathbb{T}$, possesses a simpler derivative and thus allows for a better analysis of the critical situation $3\hat{\gamma}(\theta^*)=\hat{\gamma}(\theta^*-\pi),\,\,\theta^*\in2\pi\mathbb{T}$. Therefore, inspired by their analytical framework, we adopt the $\hat{\gamma}(\theta)$ formulation to further investigate the critical situation.

The main objective of this paper is to propose a structure-preserving stabilized parametric finite element method (SPFEM) for simulating surface diffusion \eqref{eqn:surface diffusion} with the surface energy $\hat{\gamma}(\theta)$ under very mild conditions as 
\begin{enumerate}[(i)]\label{eqn:es condition gamma}
    \item $3\hat{\gamma}(\theta)\geq\hat{\gamma}(\theta-\pi),\qquad\forall\theta\in2\pi\mathbb{T},$
    \item $\hat{\gamma}^\prime(\theta^*)=0$, when $3\hat{\gamma}(\theta^*)=\hat{\gamma}(\theta^*-\pi),\qquad\theta^*\in 2\pi\mathbb{T}$.
\end{enumerate}  Compared to the $\gamma(\boldsymbol{n})$ formulation, the $\hat{\gamma}(\theta)$ formulation has the following advantages: \begin{enumerate}[(i)]
    \item it is more intuitive and has a simpler form in practical applications;
    \item it enables a reduction in the regularity requirement for the anisotropy from $C^2$ to globally $C^1$ and piecewise $C^2$;
    \item it allows for a more convenient discussion of critical situations as $3\hat{\gamma}(\theta^*)=\hat{\gamma}(\theta^*-\pi)$ for some $\theta^*\in 2\pi\mathbb{T}$ in 2D.
\end{enumerate}

The remainder of this paper is organized as follows:  In section \ref{section surface matrix}, we introduce a stabilized surface energy matrix $\hat{\boldsymbol{G}}_k(\theta)$ and propose a new conservative formulation for anisotropic surface diffusion. In section \ref{section sp-pfem}, we present a novel weak formulation based on the conservative form, introduce its spatial semi-discretization, and propose a full discretization by SPFEM.  In section \ref{section structure-preserving properties}, we analyze the structure-preserving properties of the proposed scheme, i.e., area conservation and unconditional energy stability, and develop a comprehensive framework for energy stability. It starts from defining a minimal stabilizing function $k_0(\theta)$, then we obtain the main result through a local energy estimate under the assumption that $k_0(\theta)$ is well-defined. The existence of $k_0(\theta)$ is further demonstrated in section \ref{section: existence of k0}. In section \ref{section solid-state dewetting}, we extend the SPFEM for simulating solid-state dewetting of thin films under anisotropic surface diffusion and contact line migration. Extensive numerical results are provided in section \ref{section numerical results} to demonstrate the efficiency, accuracy and structure-preserving properties of the proposed SPFEM. Finally, we conclude in section \ref{section conclusions}.

\section{A conservative formulation}\label{section surface matrix}

In this section, we propose a novel formulation with stabilization for \eqref{eqn:surface diffusion} and derive a conservative form by introducing a stabilized surface energy matrix.

Applying the geometric identity $\kappa\boldsymbol{n}=-\partial_{ss}\boldsymbol{X}$ \cite{mantegazza2011lecture}, the anisotropic surface diffusion equations \eqref{eqn:surface diffusion}--\eqref{eqn:weighted curvature} can be reformulated into \begin{subequations}\label{eqn:surface diffusion reformulate}
    \begin{align}
        &\boldsymbol{n}\cdot\partial_t\boldsymbol{X}-\partial_{ss}\mu=0,\qquad 0<s<L(t),\qquad t>0,\label{eqn:surface diffusion reformluate a}\\
        &\mu\boldsymbol{n}+\left[\hat{\gamma}(\theta)+\hat{\gamma}^{\prime\prime}(\theta)\right]\partial_{ss}\boldsymbol{X}=\boldsymbol{0}.\label{eqn:surface diffusion reformluate b}
    \end{align}
\end{subequations} where $L(t)=\int_{\Gamma(t)}\,\mathrm{d}s$ is the length of $\Gamma(t)$. 


For a vector $\boldsymbol{v}=(v_1,v_2)^T\in\mathbb{R}^2$, we denote $\boldsymbol{v}^\perp\in\mathbb{R}^2$ as its perpendicular vector which is the clockwise rotation of $\boldsymbol{v}$ by $\frac{\pi}{2}$, i.e. \begin{equation}
    \boldsymbol{v}^\perp\coloneqq(v_2,-v_1)^T=-J\boldsymbol{v},\qquad \text{with}\quad J=\left(\begin{array}{cc}
        0 & -1 \\
        1 & 0
    \end{array}\right).
\end{equation} Then the tangent vector $\boldsymbol{\tau}\coloneqq\partial_s\boldsymbol{X}$, and unit normal vector $\boldsymbol{n}$ can be written as $\boldsymbol{n}=-\boldsymbol{\tau}^\perp$. And the tangent vector can also be given by $\boldsymbol{\tau}=\boldsymbol{n}^{\perp}$. 




\begin{thm}
    For the weighted curvature $\mu$ given in \eqref{eqn:weighted curvature}, the following identity holds: \begin{equation}\label{eqn:weighted curvature alternative expression}
    \mu\boldsymbol{n}=-\partial_s\left(\hat{\boldsymbol{G}}_k(\theta)\partial_s\boldsymbol{X}\right),
\end{equation} with \begin{equation}\label{surface energy matrix}
    \hat{\boldsymbol{G}}_k(\theta)=\left(\begin{array}{cc}
       \hat{\gamma}(\theta)  & -\hat{\gamma}^\prime(\theta) \\
        \hat{\gamma}^\prime(\theta) & \hat{\gamma}(\theta)
    \end{array}\right)+k(\theta)\left(\begin{array}{cc}
     \sin^2\theta    &  -\cos\theta\sin\theta  \\
      -\cos\theta\sin\theta   &  \cos^2\theta
    \end{array}\right),\qquad \forall\theta\in2\pi\mathbb{T},
\end{equation} $k(\theta)\colon 2\pi\mathbb{T}\to\mathbb{R}_{\geq 0}$ is a non-negative stabilizing function.
\end{thm}

\begin{proof}
    Noticing \begin{equation}\label{normal vector theta formulation}
        \boldsymbol{n}=-\partial_s\boldsymbol{X}^\perp=(-\partial_sy,\partial_sx)^T,\qquad\partial_sx=\cos\theta,\qquad\partial_sy=\sin\theta,
    \end{equation} therefore \begin{equation}
        \partial_{ss}x=-\sin\theta\partial_s\theta,\qquad\partial_{ss}y=\cos\theta\partial_s\theta,
    \end{equation} which implies that \begin{equation}\label{der-kappa}
        \kappa=-(\partial_{ss}\boldsymbol{X})\cdot\boldsymbol{n}=\partial_{ss}x\partial_sy-\partial_{ss}y\partial_sx=-(\sin^2\theta+\cos^2\theta)\partial_s\theta=-\partial_s\theta.
    \end{equation} By the geometric identity $\kappa\boldsymbol{n}=-\partial_{ss}\boldsymbol{X}$ and $\partial_s\boldsymbol{X}=\tau=\boldsymbol{n}^\perp$, we obtain \begin{equation}\label{eqn:kappa*tau}
        \kappa\partial_s\boldsymbol{X}=\kappa\boldsymbol{n}^\perp=-\partial_{ss}\boldsymbol{X}^{\perp},\qquad\kappa\partial_s\boldsymbol{X}^\perp=-\kappa\boldsymbol{n}=\partial_{ss}\boldsymbol{X}.
    \end{equation} Then by \eqref{der-kappa} and \eqref{eqn:kappa*tau}, \begin{equation}\label{der gamma derX}
        \begin{aligned}
            \partial_s\Bigl(\hat{\gamma}(\theta)\partial_s\boldsymbol{X}\Bigr)&=\hat{\gamma}^\prime(\theta)\partial_s\theta\partial_s\boldsymbol{X}+\hat{\gamma}(\theta)\partial_{ss}\boldsymbol{X}\\
            &=-\kappa\hat{\gamma}^\prime(\theta)\partial_s\boldsymbol{X}+\hat{\gamma}(\theta)\partial_{ss}\boldsymbol{X}\\
            &=\hat{\gamma}^{\prime\prime}(\theta)\partial_{ss}\boldsymbol{X}^\perp+\hat{\gamma}(\theta)\partial_{ss}\boldsymbol{X},
        \end{aligned}
    \end{equation} and \begin{equation}\label{der dergamma derXperp}
        \begin{aligned}
            \partial_s\Bigl(\hat{\gamma}^\prime(\theta)\partial_s\boldsymbol{X}^\perp\Bigr)&=\hat{\gamma}^{\prime\prime}(\theta)\partial_s\theta\partial_s\boldsymbol{X}^\perp+\hat{\gamma}^\prime(\theta)\partial_{ss}\boldsymbol{X}^\perp\\
            &=-\kappa\hat{\gamma}^{\prime\prime}(\theta)\partial_s\boldsymbol{X}^\perp+\hat{\gamma}^\prime(\theta)\partial_{ss}\boldsymbol{X}^\perp\\
            &=-\hat{\gamma}^{\prime\prime}(\theta)\partial_{ss}\boldsymbol{X}+\hat{\gamma}^\prime(\theta)\partial_{ss}\boldsymbol{X}^\perp.
        \end{aligned}
    \end{equation}   
    Note that $\boldsymbol{n}^T\partial_s\boldsymbol{X}=\boldsymbol{n}\cdot\boldsymbol{\tau}\equiv0$, thus $\partial_s\left(k(\theta)\boldsymbol{n}\boldsymbol{n}^T\partial_s\boldsymbol{X}\right)$ vanishes. Combining \eqref{eqn:surface diffusion reformluate b}, \eqref{der gamma derX} and \eqref{der dergamma derXperp}, we have \begin{equation}\label{mu*normal}
    \begin{aligned}
        \mu\boldsymbol{n}&=-[\hat{\gamma}(\theta)+\hat{\gamma}^{\prime\prime}(\theta)]\partial_{ss}\boldsymbol{X}-\partial_s\left(k(\theta)\boldsymbol{n}\boldsymbol{n}^T\partial_s\boldsymbol{X}\right)\\
        &=-\partial_s\Bigl(\hat{\gamma}(\theta)\partial_s\boldsymbol{X}\Bigr)+\partial_s\Bigl(\hat{\gamma}^\prime(\theta)\partial_s\boldsymbol{X}^\perp\Bigr)-\partial_s\left(k(\theta)\boldsymbol{n}\boldsymbol{n}^T\partial_s\boldsymbol{X}\right)\\
        &=-\partial_s\Bigl(\hat{\gamma}(\theta)\partial_s\boldsymbol{X}-\hat{\gamma}^\prime(\theta)\partial_s\boldsymbol{X}^\perp+k(\theta)\boldsymbol{n}\boldsymbol{n}^T\partial_s\boldsymbol{X}\Bigr).
    \end{aligned}
\end{equation} 
On the other hand, by \eqref{normal vector theta formulation}, we have \begin{equation}
    \left(\begin{array}{cc}
     \sin^2\theta  & -\sin\theta\cos\theta \\
       -\sin\theta\cos\theta  &  \cos^2\theta
    \end{array}\right)=\left(\begin{array}{c}
         -\sin\theta  \\
         \cos\theta
    \end{array}\right)\left(-\sin\theta,\cos\theta\right)=\boldsymbol{n}\boldsymbol{n}^T,
\end{equation} consequently \begin{equation}\label{Gktau}
    \begin{aligned}
        \hat{\boldsymbol{G}}_k(\theta)\partial_s\boldsymbol{X}&=\left[\left(\begin{array}{cc}
            \hat{\gamma}(\theta) & -\hat{\gamma}^\prime(\theta) \\
            \hat{\gamma}^\prime(\theta) & \hat{\gamma}(\theta)
        \end{array}\right)+k(\theta)    \left(\begin{array}{cc}
     \sin^2\theta  & -\sin\theta\cos\theta \\
       -\sin\theta\cos\theta  &  \cos^2\theta
    \end{array}\right)\right]\partial_s\boldsymbol{X}\\
        &=\left[\hat{\gamma}(\theta)I_2+\hat{\gamma}^\prime(\theta)J+k(\theta)\boldsymbol{n}\boldsymbol{n}^T\right]\partial_s\boldsymbol{X}\\
        &=\hat{\gamma}(\theta)\partial_s\boldsymbol{X}-\hat{\gamma}^\prime(\theta)\partial_s\boldsymbol{X}^\perp+k(\theta)\boldsymbol{n}\boldsymbol{n}^T\partial_s\boldsymbol{X},
    \end{aligned}
\end{equation} where $I_2$ is the $2\times 2$ identity matrix. Substituting \eqref{Gktau} into \eqref{mu*normal}, the desired equality \eqref{eqn:weighted curvature alternative expression} is obtained. 
\end{proof}

Applying \eqref{eqn:weighted curvature alternative expression}, the governing geometric PDE \eqref{eqn:surface diffusion reformulate} for anisotropic surface diffusion can be reformulated as the following \textit{conservative form} \begin{subequations}\label{eqn:surface diffusion reformulate conservative form}
    \begin{align}
        &\boldsymbol{n}\cdot\partial_t\boldsymbol{X}-\partial_{ss}\mu=0,\label{eqn:surface diffusion reformulate conservative form a}\\
        &\mu\boldsymbol{n}+\partial_s\left(\hat{\boldsymbol{G}}_k(\theta)\partial_s\boldsymbol{X}\right)=0.\label{eqn:surface diffusion reformulate conservative form b}
    \end{align}
\end{subequations}

\begin{remark}
    If we take the stabilizing term $k(\theta)\equiv 0$ in \eqref{surface energy matrix}, then $\hat{\boldsymbol{G}}_k(\theta)=\left(\begin{array}{cc}
       \hat{\gamma}(\theta)  & -\hat{\gamma}^\prime(\theta) \\
        \hat{\gamma}^\prime(\theta) & \hat{\gamma}(\theta)
    \end{array}\right)$ collapses to the surface energy matrix $G(\theta)$ proposed in \cite{li2021energy}. Moreover, with the adoption of the $\gamma(\boldsymbol{n})$ formulation, we can define the corresponding stabilizing function $k(\boldsymbol{n})\coloneqq k(\boldsymbol{n}(\theta))=k(\theta)$ by the one-to-one correspondence $\boldsymbol{n}=\boldsymbol{n}(\theta)=(-\sin\theta, \cos\theta)^T$, and the stabilizing term is simplified to $k(\boldsymbol{n})\boldsymbol{n}\boldsymbol{n}^T$. Consequently, $\hat{\boldsymbol{G}}_k(\theta)$ is transformed into the surface energy matrix $\boldsymbol{G}_k(\boldsymbol{n})$ in \cite{bao2024structure}.
\end{remark}
\begin{remark}
    At the continuous level, $k(\theta)$ makes no contribution, as $\boldsymbol{n}^T\partial_s\boldsymbol{X} = 0$. Thus, the conservative form \eqref{eqn:surface diffusion reformulate conservative form} and the original form \eqref{eqn:surface diffusion reformulate} are equivalent. At the discrete level, however, $k(\theta)$ serves as a stabilizing term, which relaxes the energy stability conditions for the anisotropy $\hat{\gamma}(\theta)$. For example, surface matrix $G(\theta)$ in \cite{li2021energy} (absent the stabilizing term) only guarantees energy stability for specific cases of weakly anisotropic surface energy. In contrast, with this stabilizing term, this formulation can be applied to more general anisotropies, see \eqref{es condition on gamma}
\end{remark}

\section{A SPFEM for anisotropic surface diffusion}\label{section sp-pfem}

In this section, we first develop a novel weak formulation based on the conservative form \eqref{eqn:surface diffusion reformulate conservative form} and present the spatial semi-discretization of this weak formulation. After that, a structure-preserving SPFEM is proposed by adapting the implicit-explicit Euler method in time, which preserves area conservation and energy dissipation at the discrete level.

\subsection{Weak formulation}

In order to derive a weak formulation of equation \eqref{eqn:surface diffusion reformulate conservative form}, we introduce a time-independent variable $\rho$ which parameterizes $\Gamma(t)$ over a fixed domain $\rho\in\mathbb{I}=[0,1]$ as \begin{equation}
    \Gamma(t)\coloneqq\boldsymbol{X}(\rho,t)=(x(\rho,t),y(\rho,t))^T\colon\mathbb{I}\times\mathbb{R}^+\to\mathbb{R}^2.
\end{equation} The arc-length parameter $s$ can thus be computed by $s=\int_0^\rho|\partial_\rho\boldsymbol{X}(q,t)|\,\mathrm{d}q$. (We do not distinguish $\boldsymbol{X}(\rho,t)$ and $\boldsymbol{X}(s,t)$ if there's no misunderstanding.) 

Introduce the following functional space with respect to the evolution curve $\Gamma(t)$ as \begin{equation}
    L^2(\mathbb{I})\coloneqq\left\{u\colon\mathbb{I}\to\mathbb{R}\mid\int_{\Gamma(t)}|u(s)|^2\,\mathrm{d}s=\int_{\mathbb{I}}|u(s(\rho,t))|^2\partial_\rho s\,\mathrm{d}s<+\infty\right\},
\end{equation} equipped with the $L^2$-inner product \begin{equation}
    \Bigl(u,v\Bigr)_{\Gamma(t)}\coloneqq\int_{\Gamma(t)}u(s)v(s)\,\mathrm{d}s=\int_{\Gamma(t)}u(s(\rho,t))v(s(\rho,t))\partial_\rho s\,\mathrm{d}\rho,
\end{equation} for any scalar (or vector) functions. The Sobolev spaces are defined as \begin{subequations}
    \begin{align}
        H^1(\mathbb{I})&\coloneqq\left\{u\colon\mathbb{I}\to\mathbb{R}\mid u\in L^2(\mathbb{I}),\,\,\text{and}\,\,\partial_\rho u\in L^2(\mathbb{I})\right\},\\
        H^1_p(\mathbb{I})&\coloneqq\left\{u\in H^1(\mathbb{I})\mid u(0)=u(1)\right\}.
    \end{align}
\end{subequations} Extensions of above definitions to the functions in $[L^2(\mathbb{I})]^2,[H^1(\mathbb{I})]^2$ and $[H^1_p(\mathbb{I})]^2$ are straightforward.

By multiplying the equation \eqref{eqn:surface diffusion reformulate conservative form a} by a test function $\varphi\in H^1_p(\mathbb{I})$, integrating over $\Gamma(t)$, and applying integration by parts, we obtain
    \begin{equation}\label{variation1, sd}
    \Bigl(\boldsymbol{n}\cdot \partial_t\boldsymbol{X}, \varphi\Bigr)_{\Gamma(t)}+\Bigl(\partial_s\mu, \partial_s\varphi\Bigr)_{\Gamma(t)}=0.
\end{equation}

Similarly, by taking the dot product of equation \eqref{eqn:surface diffusion reformulate conservative form b} with a test function $\boldsymbol{\omega}=(\omega_1,\omega_2)^T\in[H^1_p(\mathbb{I})]^2$ and integrating by parts, we have \begin{equation}\label{variation2, sd}
    \begin{aligned}
        0&=\Bigl(\mu\boldsymbol{n}+\partial_s\left(\hat{\boldsymbol{G}}_k(\theta)\partial_s\boldsymbol{X}\right),\boldsymbol{\omega}\Bigr)_{\Gamma(t)}\\
        &=\Bigl(\mu\boldsymbol{n},\boldsymbol{\omega}\Bigr)_{\Gamma(t)}+\Bigl(\partial_s(\hat{\boldsymbol{G}}_k(\theta)\partial_s\boldsymbol{X}),\boldsymbol{\omega}\Bigr)_{\Gamma(t)}\\
        &=\Bigl(\mu\boldsymbol{n},\boldsymbol{\omega}\Bigr)_{\Gamma(t)}-\Bigl(\hat{\boldsymbol{G}}_k(\theta)\partial_s\boldsymbol{X},\partial_s\boldsymbol{\omega}\Bigr)_{\Gamma(t)}
    \end{aligned}
\end{equation}

Combining \eqref{variation1, sd} and \eqref{variation2, sd}, we propose a new weak formulation for \eqref{eqn:surface diffusion reformulate conservative form} as follows: Given an initial closed curve $\Gamma(0)\coloneqq\boldsymbol{X}(\cdot,0)=\boldsymbol{X}_0\in[H^1_p(\mathbb{I})]^2$, find the solution $(\boldsymbol{X}(\cdot,t)=(x(\cdot,t),y(\cdot,t))^T,\mu(\cdot,t))\in[H^1_p(\mathbb{I})]^2\times H^1_p(\mathbb{I})$, such that: \begin{subequations}\label{new variational formulation surface diffusion}
    \begin{align}
        \begin{split}
            &\Bigl(\boldsymbol{n}\cdot \partial_t\boldsymbol{X}, \varphi\Bigr)_{\Gamma(t)}+\Bigl(\partial_s\mu, \partial_s\varphi\Bigr)_{\Gamma(t)}=0,\qquad \forall\varphi\in H^1_p(\mathbb{I}),\label{new variational formulation surface diffusion a}
        \end{split}\\
        \begin{split}
            &\Bigl(\mu\boldsymbol{n},\boldsymbol{\omega}\Bigr)_{\Gamma(t)}-\Bigl(\hat{\boldsymbol{G}}_k(\theta)\partial_s\boldsymbol{X},\partial_s\boldsymbol{\omega}\Bigr)_{\Gamma(t)}=0,\qquad\forall\boldsymbol{\omega}\in[H^1_p(\mathbb{I})]^2.\label{new variational formulation surface diffusion b}
        \end{split}
    \end{align}
\end{subequations}

It can be demonstrated that the weak formulation \eqref{new variational formulation surface diffusion} maintains two geometric properties, namely, area conservation and energy dissipation.  

\begin{prop}[area conservation and energy dissipation]\label{prop weak form structure preserving} Suppose $\Gamma(t)$ is given by the solution $(\boldsymbol{X}(\cdot, t), \mu(\cdot, t))$ of the weak formulation \eqref{new variational formulation surface diffusion}, denote $A_c(t)$ as the enclosed area and $W_c(t)$ as the total energy of the closed evolving curve $\Gamma(t)$, respectively, which are formally given by \begin{equation}
    A_c(t)\coloneqq\int_{\Gamma(t)}y(s,t)\partial_s x(s,t)\,\mathrm{d}s,\qquad W_c(t)\coloneqq\int_{\Gamma(t)}\hat{\gamma}(\theta)\,\mathrm{d}s.
\end{equation} Then we have \begin{equation}\label{geo properties, cont, sd}
    A_c(t)\equiv A_c(0),\qquad W_c(t)\leq W_c(t_1)\leq W_c(0), \qquad \forall t\geq t_1\geq 0.
\end{equation} More precisely, \begin{equation}\label{eqn:area conservation & energy disspation of surface diffusion}
    \frac{\mathrm{d}}{\mathrm{d}t}A_c(t)=0,\qquad \frac{\mathrm{d}}{\mathrm{d}t}W_c(t)=-\int_{\Gamma(t)}|\partial_s\mu|^2\,\mathrm{d}s\leq 0,\qquad t\geq 0.
\end{equation}
\end{prop}

To prove the above theorem, we first introduce the following transport lemma:

\begin{lem}\label{lemma transport theorem}
Suppose $\Gamma(t)$ is a two-dimensional piecewise $C^1$ curve parameterized by $\boldsymbol{X}(\rho,t)$, $f\colon\Gamma(t)\times\mathbb{R}^+\to\mathbb{R}$ is a differentiable function, then 
\begin{equation}
        \frac{\mathrm{d}}{\mathrm{d}t}\int_{\Gamma(t)}f\,\mathrm{d}s=\int_{\Gamma(t)}\partial_tf+f\partial_s(\partial_t\boldsymbol{X})\cdot\partial_s\boldsymbol{X}\,\mathrm{d}s.
    \end{equation}
\end{lem}

\begin{proof}
Since $|\partial_\rho\boldsymbol{X}|=\sqrt{(\partial_\rho x)^2+(\partial_\rho y)^2}$, then \begin{equation}\label{eq: dtdrho}
    \begin{aligned}
        \partial_t|\partial_\rho\boldsymbol{X}|&=\frac{\partial_\rho x\partial_t(\partial_\rho x)+\partial_\rho y\partial_t(\partial_\rho y)}{\sqrt{(\partial_\rho x)^2+(\partial_\rho y)^2}}\\
        &=\frac{\partial_\rho\boldsymbol{X}}{|\partial_\rho\boldsymbol{X}|}\cdot\frac{\partial_\rho(\partial_t\boldsymbol{X})}{|\partial_\rho\boldsymbol{X}|}|\partial_\rho\boldsymbol{X}|\\
        &=\partial_s\boldsymbol{X}\cdot\partial_s(\partial_t\boldsymbol{X})|\partial_\rho\boldsymbol{X}|,
    \end{aligned}
\end{equation} thus 
    \begin{equation}
        \begin{aligned}
            \frac{\mathrm{d}}{\mathrm{d}t}\int_{\Gamma(t)}f\,\mathrm{d}s&=\frac{\mathrm{d}}{\mathrm{d}t}\int_0^1f|\partial_\rho\boldsymbol{X}|\,\mathrm{d}\rho\\
            &=\int_{0}^1\partial_tf\,|\partial_\rho\boldsymbol{X}|+f\partial_t|\partial_\rho\boldsymbol{X}|\,\mathrm{d}\rho\\
            &=\int_0^1\partial_tf\,|\partial_\rho\boldsymbol{X}|+f\partial_s(\partial_t\boldsymbol{X})\cdot\partial_s\boldsymbol{X}|\partial_\rho\boldsymbol{X}|\,\mathrm{d}\rho\\
            &=\int_{\Gamma(t)}\partial_tf+f\partial_s(\partial_t\boldsymbol{X})\cdot\partial_s\boldsymbol{X}\,\mathrm{d}s.
        \end{aligned}
    \end{equation}
\end{proof}

Now the proof of Proposition \ref{prop weak form structure preserving} is ready to be presented: \begin{proof}
    Denote the region enclosed by $\Gamma(t)$ as $\Omega(t)$. For the area conservation, by the Reynolds' transport theorem \cite{reynolds1983papers} and taking $\varphi\equiv 1$ in \eqref{new variational formulation surface diffusion a}, \begin{equation}
        \begin{aligned}
            \frac{\mathrm{d}}{\mathrm{d}t}A_c(t)&=\frac{\mathrm{d}}{\mathrm{d}t}\int_{\Omega(t)}1\,\mathrm{d}x\mathrm{d}y=\int_{\Gamma(t)}\boldsymbol{n}\cdot\partial_t\boldsymbol{X}\,\mathrm{d}s\\
            &=\Bigl(\boldsymbol{n}\cdot\partial_t\boldsymbol{X},1\Bigr)_{\Gamma(t)}=-\Bigl(\partial_s\mu,\partial_s1\Bigr)_{\Gamma(t)}=0.\\
        \end{aligned}
    \end{equation} 
    For the energy dissipation part, by Lemma~\ref{lemma transport theorem}, we have \begin{equation}
    \begin{aligned}
        \frac{\mathrm{d}}{\mathrm{d}t}W_c(t)&=\int_{\Gamma(t)}\partial_t\hat{\gamma}(\theta)+\hat{\gamma}(\theta)\partial_s(\partial_t\boldsymbol{X})\cdot\partial_s\boldsymbol{X}\,\mathrm{d}s\\
        &=\int_{\Gamma(t)}\hat{\gamma}^\prime(\theta)\partial_t\theta+\hat{\gamma}(\theta)\partial_s(\partial_t\boldsymbol{X})\cdot\partial_s\boldsymbol{X}\,\mathrm{d}s.
    \end{aligned}
    \end{equation}
    On the other hand, by using \eqref{eq: dtdrho}, we can simplify $\partial_s (\partial_t \boldsymbol{X})$ as
    \begin{align}
        \partial_s (\partial_t \boldsymbol{X})&= \frac{1}{|\partial_\rho \boldsymbol{X}|} \partial_\rho (\partial_t \boldsymbol{X})\nonumber\\
        &= \frac{1}{|\partial_\rho \boldsymbol{X}|} \partial_t \left(|\partial_\rho \boldsymbol{X}|(\cos\theta, \sin\theta)^T\right)\\
        &= \frac{1}{|\partial_\rho \boldsymbol{X}|} \left(\partial_s \boldsymbol{X}\cdot \partial_s (\partial_t \boldsymbol{X})|\partial_\rho \boldsymbol{X}|(\cos\theta, \sin\theta)^T + |\partial_\rho \boldsymbol{X}|(-\sin\theta, \cos\theta)^T\partial_t \theta\right)\nonumber
    \end{align}
    This, together with the fact $\partial_s\boldsymbol{X}^\perp=(\sin\theta,-\cos\theta)^T$, yields that \begin{equation}
        \partial_t\theta=-\partial_s(\partial_t\boldsymbol{X})\cdot\partial_s\boldsymbol{X}^\perp.
    \end{equation} Therefore, \begin{equation}
        \begin{aligned}
            \frac{\mathrm{d}}{\mathrm{d}t}W_c(t)&=\int_{\Gamma(t)}\left[\hat{\gamma}(\theta)\partial_s\boldsymbol{X}-\hat{\gamma}^\prime(\theta)\partial_s\boldsymbol{X}^\perp\right]\cdot\partial_s(\partial_t\boldsymbol{X})\,\mathrm{d}s.
        \end{aligned}
    \end{equation} Since $\boldsymbol{n}^T\partial_s\boldsymbol{X}\equiv 0$, then \begin{equation}
    \begin{aligned}
        \hat{\gamma}(\theta)\partial_s\boldsymbol{X}-\hat{\gamma}^\prime(\theta)\partial_s\boldsymbol{X}^\perp&=\hat{\gamma}(\theta)\partial_s\boldsymbol{X}-\hat{\gamma}^\prime(\theta)\partial_s\boldsymbol{X}^\perp+k(\theta)\boldsymbol{n}\boldsymbol{n}^T\partial_s\boldsymbol{X}\\
        &=\hat{\boldsymbol{G}}_k(\theta)\partial_s\boldsymbol{X}
    \end{aligned}
    \end{equation} which leads to \begin{equation}
        \frac{\mathrm{d}}{\mathrm{d}t}W_c(t)=\int_{\Gamma(t)}\hat{\boldsymbol{G}}_k(\theta)\partial_s\boldsymbol{X}\cdot\partial_s(\partial_t\boldsymbol{X})\,\mathrm{d}s.
    \end{equation} 
    Therefore, by taking $\varphi=\mu$ and $\boldsymbol{\omega}=\partial_t\boldsymbol{X}$ in \eqref{new variational formulation surface diffusion a} and \eqref{new variational formulation surface diffusion b}, respectively, we have \begin{equation}
        \begin{aligned}
            \frac{\mathrm{d}}{\mathrm{d}t}W_c(t)&=\Bigl(\hat{\boldsymbol{G}}_k(\theta)\partial_s\boldsymbol{X},\partial_s(\partial_t\boldsymbol{X})\Bigr)_{\Gamma(t)}\\
            &=\Bigl(\mu\boldsymbol{n},\partial_t\boldsymbol{X}\Bigr)_{\Gamma(t)}=-\Bigl(\partial_s\mu,\partial_s\mu\Bigr)_{\Gamma(t)}\leq 0.
        \end{aligned}
    \end{equation}
\end{proof}

\subsection{A semi-discretization in space}

To obtain the spatial discretization, let $N>2$ be a positive integer and $h=1/N$ be the mesh size, grid points $\rho_j=jh$, sub-intervals $I_j=[\rho_{j-1},\rho_j]$ for $j=1,2,\dots,N$ and the uniform partition $\mathbb{I}=[0,1]=\cup_{j=1}^NI_j$. The closed
 curve $\Gamma(t)=\boldsymbol{X}(\cdot,t)$ is approximated by the  polygonal curve $\Gamma^h(t)=\boldsymbol{X}^h(\cdot,t)=\left(x^h(\cdot,t),y^h(\cdot,t)\right)^T$ satisfying $\boldsymbol{X}^h(\rho_j,0)=\boldsymbol{X}(\rho_j,0)$.

The polygon $\Gamma^h(t)$ is composed of ordered line segments $\{\boldsymbol{h}_j(t)\}_{j=1}^N$, i.e. \begin{equation}
    \Gamma^h(t)=\bigcup_{j=1}^N\boldsymbol{h}_j(t) \qquad\text{with}\quad\boldsymbol{h}_j(t)=(h_{j,x},h_{j,y})^T\coloneqq\boldsymbol{X}^h(\rho_j,t)-\boldsymbol{X}^h(\rho_{j-1}, t).
\end{equation} And we always assume that $\displaystyle h_{\min}(t)=\min_{1\leq j\leq N}|\boldsymbol{h}_j(t)|>0$ for all $t>0$. 

By using $\boldsymbol{h}_j$, the discrete geometric quantities such as the unit tangential vector $\boldsymbol{\tau}^h$, the outward unit normal vector $\boldsymbol{n}^h$ and the inclination angle $\theta^h$ can be computed on each segment as: \begin{equation}
    \boldsymbol{\tau}^h|_{I_j}=\frac{\boldsymbol{h}_j}{|\boldsymbol{h}_j|}\coloneqq\boldsymbol{\tau}_j^h,\qquad\boldsymbol{n}^h|_{I_i}=-(\boldsymbol{\tau}_j^h)^\perp=-\frac{(\boldsymbol{h}_j)^\perp}{|\boldsymbol{h}_j|}\coloneqq\boldsymbol{n}_j^h;
\end{equation} and \begin{equation}
    \theta^h|_{I_j}\coloneqq\theta_j^h,\qquad\text{satisfying}\quad\cos\theta^h_j=\frac{h_{j,x}}{|\boldsymbol{h}_j|},\quad\sin\theta_j^h=\frac{h_{j,y}}{|\boldsymbol{h}_j|}.
\end{equation}

We introduce the finite element subspaces \begin{subequations}
    \begin{align}
        &\mathbb{K}^h\coloneqq\left\{u^h\in C(\mathbb{I})\mid u^h|_{I_j}\in\mathcal{P}^1(I_j),\,\,\forall j=1,2,\dots,N\right\}\subseteq H^1(\mathbb{I}),\\
        &\mathbb{K}_p^h\coloneqq\{u^h\in\mathbb{K}^h\mid u^h(0)=u^h(1)\},\qquad\mathbb{X}_p^h\coloneqq[H_p^1(\mathbb{I})]^2,
    \end{align}
\end{subequations} where $\mathcal{P}^1(I_j)$ is the set of polynomials defined on $I_j$ of degree $\leq 1$. For $u,v\in\mathbb{K}^h$, the mass-lumped inner product $\Bigl(\cdot,\cdot\Bigr)^h_{\Gamma^h(t)}$ with respect to $\Gamma^h(t)$ is defined as \begin{equation}\label{eqn:mass lumped inner product}
    \Bigl(u, v\Bigr)^h_{\Gamma^h(t)}\coloneqq\frac{1}{2}\sum_{j=1}^N |\boldsymbol{h}_j(t)|\,\left((u\cdot v)(\rho_{j-1}^+)+(u\cdot v)(\rho_j^-)\right).
\end{equation}
where $u(\rho_j^\pm)=\lim\limits_{\rho\to \rho_j^\pm} u(\rho)$. And the discretized differential operator $\partial_s$ for $f\in\mathbb{K}^h$ is defined as \begin{equation}
    \partial_s f|_{I_j}\coloneqq\frac{f(\rho_j)-f(\rho_{j-1})}{|\boldsymbol{h}_j|}.
\end{equation} 
The above definitions also hold true for vector-valued functions. 

We now propose the spatial semi-discretization for \eqref{new variational formulation surface diffusion} as follows: Let $\Gamma_0^h\coloneqq \boldsymbol{X}^h(\cdot,0)\in\mathbb{X}^h_p,\mu(\cdot)\in\mathbb{K}^h_p$ be the approximations of $\Gamma_0\coloneqq \boldsymbol{X}_0(\cdot),\mu_0(\cdot)$, respectively, for $t>0$, find the solution $\left(\boldsymbol{X}^h(\cdot,t),\mu^h(\cdot)\right)\in\mathbb{X}_p^h\times\mathbb{K}_p^h$ such that \begin{subequations}\label{spatial semi-discretization surface diffusion}
    \begin{align}
            &\Bigl(\boldsymbol{n}^h\cdot \partial_t\boldsymbol{X}^h, \varphi^h\Bigr)_{\Gamma^h(t)}^h+\Bigl(\partial_s\mu^h, \partial_s\varphi^h\Bigr)_{\Gamma^h(t)}^h=0,\qquad \forall\varphi^h\in \mathbb{K}_p^h,\\
            &\Bigl(\mu^h\boldsymbol{n}^h,\boldsymbol{\omega}^h\Bigr)_{\Gamma^h(t)}^h-\Bigl(\hat{\boldsymbol{G}}_k(\theta^h)\partial_s\boldsymbol{X}^h,\partial_s\boldsymbol{\omega}^h\Bigr)_{\Gamma^h(t)}^h=0,\qquad\forall\boldsymbol{\omega}^h\in\mathbb{X}_p^h,
    \end{align}
\end{subequations} where \begin{equation}
    \hat{\boldsymbol{G}}_k(\theta^h)=\left(\begin{array}{cc}
        \hat{\gamma}(\theta^h) & -\hat{\gamma}^\prime(\theta^h)\\
        \hat{\gamma}^\prime(\theta^h) & \hat{\gamma}(\theta^h)
    \end{array}\right)+k(\theta^h)\left(\begin{array}{cc}
        \sin^2\theta^h & -\cos\theta^h\sin\theta^h\\
        -\cos\theta^h\sin\theta^h & \cos^2\theta^h
    \end{array}\right).
\end{equation}

Denote the enclosed area and the free energy of the polygonal curve $\Gamma^h(t)$ as $A^h_c(t)$ and  $W^h_c(t)$, respectively, which are given by \begin{subequations}\label{discrete area}
    \begin{align}
        &A^h_c(t)\coloneqq\frac{1}{2}\sum_{j=1}^N(x_j^h(t)-x_{j-1}^h(t))(y_j^h(t)+y_{j-1}^h(t)),\\
        &W^h_c(t)\coloneqq\sum_{j=1}^N|\boldsymbol{h}_j(t)|\hat{\gamma}(\theta^h_j),
    \end{align}
\end{subequations}
where $x_j^h(t)\coloneqq x^h(\rho_j, t), y_j^h(t)\coloneqq y^h(\rho_j, t),\,\, \forall 0\leq j\leq N$. 

Following similar steps in Proposition \ref{prop weak form structure preserving}, it can be proved that the two geometric properties for the semi-discretization \eqref{spatial semi-discretization surface diffusion} still preserves: \begin{prop}[area conservation and energy dissipation] Suppose $\Gamma^h(t)$ is given by the solution $(\boldsymbol{X}^h(\cdot, t), \mu^h(\cdot, t))$ of \eqref{spatial semi-discretization surface diffusion}, then we have
\begin{equation}
    A^h_c(t)\equiv A^h_c(0),\quad W^h_c(t)\leq W^h_c(t_1)\leq W^h_c(0), \quad \forall t\geq t_1\geq 0.
\end{equation}
\end{prop}

\subsection{A structure-preserving SPFEM discretization}

Let $\tau$ be the uniform time step. Denoting the approximation of $\Gamma(t)=\boldsymbol{X}(\cdot,t)$ at $t_m=m\tau, m=0,1,\dots,$ as $\Gamma^m=\boldsymbol{X}^m(\cdot)=\cup_{j=1}^N\boldsymbol{h}_j^m$ where $\boldsymbol{h}_j^m\coloneqq\boldsymbol{X}^m(\rho_j)-\boldsymbol{X}^m(\rho_{j-1})$. Then the definitions of the mass lumped inner product $(\cdot,\cdot)^h_{\Gamma^m}$, the unit tangential vector $\boldsymbol{\tau}^m$, the unit outward normal vector $\boldsymbol{n}^m$, and the inclination angle $\theta^m$ with respect to $\Gamma^m$ can be given in a similar approach.

Following the ideas in \cite{bao2021structure, jiang2021perimeter,bao2023symmetrized2D, bao2023structure} to design an SP-PFEM for surface diffusion, we utilize the explicit-implicit Euler method in time. The derived fully-implicit structure-preserving discretization of SPFEM for the anisotropic surface diffusion \eqref{eqn:surface diffusion reformulate} is expressed as follows: 

Suppose the initial approximation $\Gamma^0(\cdot)\in \mathbb{X}^h$ is given by $\boldsymbol{X}^0(\rho_j)=\boldsymbol{X}_0(\rho_j), \forall 0\leq j\leq N$. For any $m=0, 1, 2, \ldots$, find the solution $(\boldsymbol{X}^{m+1}(\cdot)=(x^{m+1}(\cdot),y^{m+1}(\cdot))^T, \mu^{m+1}(\cdot))\in \mathbb{X}^h_p\times \mathbb{K}^h_p$ such that \begin{subequations}\label{eqn:sp-pfem surface diffusion}
    \begin{align}
            &\Bigl(\boldsymbol{n}^{m+\frac{1}{2}}\cdot \partial_t\boldsymbol{X}^{m+1}, \varphi^h\Bigr)_{\Gamma^m}^h+\Bigl(\partial_s\mu^{m+1}, \partial_s\varphi^h\Bigr)_{\Gamma^m}^h=0,\qquad \forall\varphi^h\in \mathbb{K}_p^h,\\
            &\Bigl(\mu^{m+1}\boldsymbol{n}^{m+\frac{1}{2}},\boldsymbol{\omega}^h\Bigr)_{\Gamma^m}^h-\Bigl(\hat{\boldsymbol{G}}_k(\theta^m)\partial_s\boldsymbol{X}^{m+1},\partial_s\boldsymbol{\omega}^h\Bigr)_{\Gamma^m}^h=0,\qquad\forall\boldsymbol{\omega}^h\in\mathbb{X}_p^h,
    \end{align}
\end{subequations} where \begin{equation}
    \hat{\boldsymbol{G}}_k(\theta^m)=\left(\begin{array}{cc}
        \hat{\gamma}(\theta^m) & -\hat{\gamma}^\prime(\theta^m)\\
        \hat{\gamma}^\prime(\theta^m) & \hat{\gamma}(\theta^m)
    \end{array}\right)+k(\theta^m)\left(\begin{array}{cc}
        \sin^2\theta^m & -\cos\theta^m\sin\theta^m\\
        -\cos\theta^m\sin\theta^m & \cos^2\theta^m
    \end{array}\right),
\end{equation} and \begin{equation}\label{dis surface energy matrix}
    \boldsymbol{n}^{m+\frac{1}{2}}\coloneqq -\frac{1}{2}\frac{1}{|\partial_\rho \boldsymbol{X}^m|}(\partial_\rho \boldsymbol{X}^m+\partial_\rho\boldsymbol{X}^{m+1})^\perp.
\end{equation}

\begin{remark}
    The above scheme is weakly implicit, as the integral domain is explicitly chosen and each equation contains only one non-linear term. The nonlinear term is a polynomial function of degree $\leq 2$ with respect to the components of $\boldsymbol{X}^{m+1}$ and $\mu^{m+1}$, thus it can be efficiently and accurately solved by the Newton's iterative method similar to \cite{bao2021structure}.
\end{remark}

\begin{remark}
    The choice of $\boldsymbol{n}^{m+\frac{1}{2}}$ is crucial for maintaining the area conservation. The scheme becomes semi-implicit if $\boldsymbol{n}^{m+\frac{1}{2}}$ is replaced by $\boldsymbol{n}^m$, and only the energy dissipation property is preserved.
\end{remark}

\subsection{Main results}

Denote the enclosed area and the free energy of the polygon $\Gamma^m$ as $A^m_c$ and  $W^m_c$, respectively, which are given by
\begin{subequations}\label{fully discrete area and energy}
\begin{align}
\label{fully discrete area and energy, area}
A^m_c&\coloneqq\frac{1}{2}\sum_{j=1}^N\left(x^m(\rho_j)-x^m(\rho_{j-1})\right)\left(y^m(\rho_j)
+y^m(\rho_{j-1})\right),\\
\label{fully discrete area and energy, energy}
W^m_c&\coloneqq\sum_{j=1}^N|\boldsymbol{h}_j^m|\hat{\gamma}(\theta^m_j).
\end{align}
\end{subequations} 

In practical applications, it's common to encounter situations where $\hat{\gamma}(\theta)$ lacks high regularity. In the following sections, we always assume that $\hat{\gamma}(\theta)$ is globally $C^1$ and piecewise $C^2$ on $2\pi\mathbb{T}$. 



We thus introduce the following energy stable conditions on $\hat{\gamma}(\theta)$:
\begin{defi}[energy stable condition]\label{def anisotropic stable}
Suppose $\hat{\gamma}(\theta)$ is globally $C^1$ and piecewise $C^2$, the energy stable conditions on $\hat{\gamma}(\theta)$ are given as follows:
\begin{subequations}\label{es condition on gamma}
    \begin{align}
        &3\hat{\gamma}(\theta)\geq\hat{\gamma}(\theta-\pi),\qquad\forall\theta\in 2\pi\mathbb{T},\\
        &\hat{\gamma}^\prime(\theta^*)=0,\,\,\text{when}\,\,3\hat{\gamma}(\theta^*)=\hat{\gamma}(\theta^*-\pi),\qquad\theta^*\in 2\pi\mathbb{T}.
    \end{align}
\end{subequations}
\end{defi}

The main result of this work is the following structure-preserving property of the SPFEM \eqref{eqn:sp-pfem surface diffusion}:

\begin{thm}[structure-preserving]\label{MAIN RESULT}
    For any $\hat{\gamma}(\theta)$ satisfying \eqref{es condition on gamma}, the SPFEM \eqref{eqn:sp-pfem surface diffusion} is area conservative and unconditional energy dissipative with sufficiently large $k(\theta)$, i.e. \begin{equation}\label{thm:main results}
        A^{m+1}_c=A^m_c=\cdots=A^0_c,\qquad W^{m+1}_c\leq W^m_c\leq\cdots\leq W^0_c,\qquad\forall m\geq 0.
    \end{equation}
\end{thm}

The proof of area conservation part is analogous to \cite[Theorem~2.1]{bao2021structure}, we omit here for brevity. And the energy dissipation will be established in the next section.

\begin{remark}\label{kfold remark}
    For the $m$-fold anisotropy \eqref{eqn:k-fold anisotropy}: the energy stable condition \eqref{es condition on gamma} holds when $|\beta|<1$ and $|\beta|\leq \frac{1}{2}$ for $m$ being even and odd, respectively. It is a significant improvement compared to the energy stable condition $|\beta|\leq\frac{1}{m^2+1}$ in \cite{li2021energy}.
\end{remark}
\begin{remark}\label{ellipsoidal remark}
    For the ellipsoidal anisotropy \eqref{eqn:ellipsoidal anisotropy}: the energy stable condition in \cite{li2021energy} requires $-\frac{a}{2}\leq b\leq a$; while condition \eqref{es condition on gamma} is satisfied for any $a>0,a+b>0$.
\end{remark}
\begin{remark}\label{symmetric remark}
    It is noteworthy that for any symmetric anisotropy $\hat{\gamma}(\theta)$ satisfying $\hat{\gamma}(\theta)=\hat{\gamma}(\theta-\pi),\,\,\forall\theta\in 2\pi\mathbb{T}$, such as the Riemannian-like metric anisotropy \eqref{eqn:BNG anisotropy}, condition \eqref{es condition on gamma} naturally holds. Thereby it ensures the unconditional energy stability of the SPFEM \eqref{eqn:sp-pfem surface diffusion}.
\end{remark}

\section{Unconditionally energy stability of the SPFEM}\label{section structure-preserving properties}

The key point in proving energy dissipation of \eqref{eqn:sp-pfem surface diffusion} is to establish an energy estimate akin to \begin{equation}
    \left(\hat{\boldsymbol{G}}_k(\theta^m)\partial_s\boldsymbol{X}^{m+1},\partial_s(\boldsymbol{X}^{m+1}-\boldsymbol{X}^m)\right)_{\Gamma^m}^h\geq W_c^{m+1}-W_c^m,
\end{equation}  for controlling the energy difference between two subsequent time steps with the surface energy matrix $\hat{\boldsymbol{G}}_k$. To achieve desired inequality, we need a local version of the estimate, which is formulated by the following lemma: \begin{lem}[local energy estimate]\label{lemma:local energy estimate}
    Suppose $\boldsymbol{h},\hat{\boldsymbol{h}}$ are two non-zero vectors in $\mathbb{R}^2$. Let $\boldsymbol{n}=-\frac{\boldsymbol{h}^\perp}{|\boldsymbol{h}|}=(-\sin\theta,\cos\theta)^T$ and $\hat{\boldsymbol{n}}=-\frac{\hat{\boldsymbol{h}}^\perp}{|\hat{\boldsymbol{h}}|}=(-\sin\hat{\theta},\cos\hat{\theta})^T$ be the corresponding unit normal vectors. Then for sufficiently large $k(\theta)$, the following inequality holds
    \begin{equation}\label{eqn:local energy estimate}
        \frac{1}{|\boldsymbol{h}|}\left(\hat{\boldsymbol{G}}_k(\theta)\hat{\boldsymbol{h}}\right)\cdot(\hat{\boldsymbol{h}}-\boldsymbol{h})\geq|\hat{\boldsymbol{h}}|\hat{\gamma}(\hat{\theta})-|\boldsymbol{h}|\hat{\gamma}(\theta).
    \end{equation}
\end{lem}

\begin{remark}
    It is noteworthy to mention that the condition \eqref{es condition on gamma} is almost sufficient to the local energy estimate. Let $\hat{\boldsymbol{h}}=-\boldsymbol{h}$ and $\hat{\theta}=\theta-\pi$. Consequently, the inequality \eqref{eqn:local energy estimate} becomes $2|\boldsymbol{h}|\hat{\gamma}(\theta)\geq |\boldsymbol{h}|\hat{\gamma}(\theta-\pi)-|\boldsymbol{h}|\hat{\gamma}(\theta)$. Therefore, our energy stability condition \eqref{es condition on gamma} proves to be almost essential for the local energy estimate.
\end{remark}

\subsection{The minimal stabilizing function and its properties}

To prove the local energy estimate \eqref{eqn:local energy estimate}, the following two auxiliary functions are introduced as: \begin{subequations}\label{eqn:auxiliary functions}
    \begin{align}
        &P_\alpha(\phi,\theta)\coloneqq 2\sqrt{\hat{\gamma}^2(\theta)+\alpha\hat{\gamma}(\theta)\sin^2\phi},\,\,\,\,\,\quad\quad\quad\qquad\forall\phi\in2\pi\mathbb{T},\\
        &Q(\phi,\theta)\coloneqq\hat{\gamma}(\theta-\phi)+\hat{\gamma}(\theta)\cos\phi+\hat{\gamma}^\prime(\theta)\sin\phi,\qquad\forall\phi\in2\pi\mathbb{T}.
    \end{align}
\end{subequations} With the help of $P_\alpha,Q$, we present the definition of minimal stabilizing function as follows: \begin{equation}\label{eqn:def minimal stabilizing function}
    k_0(\theta)\coloneqq\inf\left\{\alpha\geq 0\mid P_\alpha(\phi,\theta)-Q(\phi,\theta)\geq 0,\,\,\forall\phi\in 2\pi\mathbb{T}\right\}.
\end{equation}

The following theorem guarantees the existence of $k_0(\theta)$: \begin{thm}\label{thm:existence of k0}
    For $\hat{\gamma}(\theta)$ satisfying \eqref{es condition on gamma}, the minimal stabilizing function $k_0(\theta)$, as given in \eqref{eqn:def minimal stabilizing function}, is well-defined.
\end{thm}

The proof of Theorem~\ref{thm:existence of k0} will be presented in Section \ref{section: existence of k0}.\\

Once the $\hat{\gamma}(\theta)$ is given, the minimal stabilizing function $k_0(\theta)$ is determined, inducing a mapping from $\hat{\gamma}(\theta)$ to $k_0(\theta)$. Moreover, this mapping is sublinear, i.e., it is positive homogeneity and subadditivity.

\begin{lem}[positive homogeneity and subadditivity]
    Let $\hat{\gamma}_i(\theta)$ for $i=1,2,3$ denote anisotropies satisfying \eqref{es condition on gamma}, each corresponding to its minimal stabilizing functions $k_i(\theta)$. Then, we have
\begin{enumerate}[(i)]
    \item if $\hat{\gamma}_2(\theta)=c\hat{\gamma}_1(\theta)$ for a given positive number $c>0$, then $k_2(\theta)=ck_1(\theta)$;\\
    \item if $\hat{\gamma}_3(\theta)=\hat{\gamma}_1(\theta)+\hat{\gamma}_2(\theta)$, then $k_3(\theta)\leq k_1(\theta)+k_2(\theta)$.
\end{enumerate}
\end{lem} The proof is similar to \cite[Lemma~4.4]{bao2023symmetrized2D},  we omit here for brevity.

\subsection{Proof of the local energy estimate}

\begin{proof}
    Applying the definitions of $\hat{\boldsymbol{G}}_k(\theta)$ in \eqref{surface energy matrix}, noting that \begin{equation}
        \left(\begin{array}{cc}
        \sin^2\theta & -\cos\theta\sin\theta\\
         -\cos\theta\sin\theta & \cos^2\theta
        \end{array}\right)=\left(\begin{array}{c}
             -\sin\theta  \\
             \cos\theta
        \end{array}\right)(-\sin\theta,\cos\theta)=\boldsymbol{n}\boldsymbol{n}^T,
    \end{equation} then we have \begin{equation}
        \begin{aligned}
            \frac{1}{|\boldsymbol{h}|}\left(\hat{\boldsymbol{G}}_k(\theta)\hat{\boldsymbol{h}}\right)\cdot\hat{\boldsymbol{h}}&=\frac{1}{|\boldsymbol{h}|}\left[\hat{\gamma}(\theta)I_2+\hat{\gamma}^{\prime}(\theta)J+k(\theta)\boldsymbol{n}\boldsymbol{n}^T\right]\hat{\boldsymbol{h}}\cdot\hat{\boldsymbol{h}}\\
            &=\frac{1}{|\boldsymbol{h}|}\hat{\gamma}(\theta)|\hat{\boldsymbol{h}}|^2+\frac{1}{|\boldsymbol{h}|}\hat{\gamma}^\prime(\theta)J\hat{\boldsymbol{h}}\cdot\hat{\boldsymbol{h}}+\frac{1}{|\boldsymbol{h}|}k(\theta)(\boldsymbol{n}\cdot\hat{\boldsymbol{h}})^2\\
            &=\frac{1}{|\boldsymbol{h}|}\hat{\gamma}(\theta)|\hat{\boldsymbol{h}}|^2+\frac{1}{|\boldsymbol{h}|}k(\theta)(\boldsymbol{n}\cdot\hat{\boldsymbol{h}})^2\\
            &=\frac{1}{|\boldsymbol{h}|}\left[\hat{\gamma}(\theta)+k(\theta)\sin^2(\theta-\hat{\theta})\right]|\hat{\boldsymbol{h}}|^2,
        \end{aligned}
    \end{equation} and \begin{equation}
        \begin{aligned}
            \frac{1}{|\boldsymbol{h}|}\left(\hat{\boldsymbol{G}}_k(\theta)\hat{\boldsymbol{h}}\right)\cdot\boldsymbol{h}&=\frac{1}{|\boldsymbol{h}|}\left[\hat{\gamma}(\theta)I_2+\hat{\gamma}^{\prime}(\theta)J+k(\theta)\boldsymbol{n}\boldsymbol{n}^T\right]\hat{\boldsymbol{h}}\cdot\boldsymbol{h}\\
            &=\frac{1}{|\boldsymbol{h}|}\hat{\gamma}(\theta)(\boldsymbol{h}\cdot\hat{\boldsymbol{h}})+\frac{1}{|\boldsymbol{h}|}\hat{\gamma}^\prime(\theta)J\hat{\boldsymbol{h}}\cdot\boldsymbol{h}+\frac{1}{|\boldsymbol{h}|}k(\theta)\boldsymbol{n}\boldsymbol{n}^T\hat{\boldsymbol{h}}\cdot\boldsymbol{h}\\
            &=|\hat{\boldsymbol{h}}|\hat{\gamma}(\theta)\left(\begin{array}{c}
                 \cos\theta  \\
                 \sin\theta
            \end{array}\right)\cdot\left(\begin{array}{c}
                 \cos\hat{\theta}  \\
                 \sin\hat{\theta}
            \end{array}\right)+|\hat{\boldsymbol{h}}|\hat{\gamma}^\prime(\theta)\left(\begin{array}{cc}
              0   & -1 \\
               1  &  0
            \end{array}\right)\left(\begin{array}{c}
                 \cos\hat{\theta}  \\
                 \sin\hat{\theta}
            \end{array}\right)\cdot\left(\begin{array}{c}
                 \cos\theta  \\
                 \sin\theta
            \end{array}\right)\\
            &\quad+|\hat{\boldsymbol{h}}|k(\theta)\left(\begin{array}{cc}
                \sin^2\theta & -\cos\theta\sin\theta \\
                -\cos\theta\sin\theta & \cos^2\theta
            \end{array}\right)\left(\begin{array}{c}
                 \cos\hat{\theta}  \\
                 \sin\hat{\theta}
            \end{array}\right)\cdot\left(\begin{array}{c}
                 \cos\theta  \\
                 \sin\theta
            \end{array}\right)\\
            &=|\hat{\boldsymbol{h}}|\hat{\gamma}(\theta)\cos(\theta-\hat{\theta})+|\hat{\boldsymbol{h}}|\hat{\gamma}^\prime(\theta)(-\sin\hat{\theta},\cos\hat{\theta})\cdot(\cos\theta,\sin\theta)\\
            &=|\hat{\boldsymbol{h}}|\left[\hat{\gamma}(\theta)\cos(\theta-\hat{\theta})+\hat{\gamma}^\prime(\theta)\sin(\theta-\hat{\theta})\right].
        \end{aligned}
    \end{equation} Recall the definitions of $P_{\alpha}(\phi,\theta),Q(\phi,\theta)$ in \eqref{eqn:auxiliary functions}, we have \begin{subequations}
        \begin{align}
            \frac{1}{|\boldsymbol{h}|}\left(\hat{\boldsymbol{G}}_k(\theta)\hat{\boldsymbol{h}}\right)\cdot\hat{\boldsymbol{h}}&=\frac{|\hat{\boldsymbol{h}}|^2}{4|\boldsymbol{h}|\hat{\gamma}(\theta)}P_{k(\theta)}^2(\phi,\theta),\label{local estimate 1}\\
            \frac{1}{|\boldsymbol{h}|}\left(\hat{\boldsymbol{G}}_k(\theta)\hat{\boldsymbol{h}}\right)\cdot\boldsymbol{h}&=|\hat{\boldsymbol{h}}|(Q(\phi,\theta)-\hat{\gamma}(\theta-\phi))\label{local estimate 2},
        \end{align}
    \end{subequations} with $\phi=\theta-\hat{\theta}$.

    Substituting \eqref{local estimate 1}, \eqref{local estimate 2} into the local energy estimate \eqref{eqn:local energy estimate}, we have \begin{equation}
        \begin{aligned}
            \frac{1}{|\boldsymbol{h}|}&\left(\hat{\boldsymbol{G}}_k(\theta)\hat{\boldsymbol{h}}\right)\cdot(\hat{\boldsymbol{h}}-\boldsymbol{h})-\left(|\hat{\boldsymbol{h}}|\hat{\gamma}(\hat{\theta})-|\boldsymbol{h}|\hat{\gamma}(\theta)\right)\\
            &=\frac{|\hat{\boldsymbol{h}}|^2}{4|\boldsymbol{h}|\hat{\gamma}(\theta)}P_{k(\theta)}^2(\phi,\theta)-|\hat{\boldsymbol{h}}|(Q(\phi,\theta)-\hat{\gamma}(\theta-\phi))\\
            &\quad-|\hat{\boldsymbol{h}}|\hat{\gamma}(\theta-\phi)+|\boldsymbol{h}|\hat{\gamma}(\theta)\\
            &=\frac{1}{4|\boldsymbol{h}|\hat{\gamma}(\theta)}\left[|\hat{\boldsymbol{h}}|^2P_{k(\theta)}^2(\phi,\theta)-4|\hat{\boldsymbol{h}}||\boldsymbol{h}|\hat{\gamma}(\theta)Q(\phi,\theta)+4|\boldsymbol{h}|^2\hat{\gamma}^2(\theta)\right],
        \end{aligned}
    \end{equation} thus the local energy estimate \eqref{eqn:local energy estimate} is equivalent to \begin{equation}
        |\hat{\boldsymbol{h}}|^2P_{k(\theta)}^2(\phi,\theta)-4|\hat{\boldsymbol{h}}||\boldsymbol{h}|\hat{\gamma}(\theta)Q(\phi,\theta)+4|\boldsymbol{h}|^2\hat{\gamma}^2(\theta)\geq 0.
    \end{equation} i.e. \begin{equation}\label{estimate finial equation}
        \left(|\hat{\boldsymbol{h}}|P_{k(\theta)}-2|\boldsymbol{h}|\hat{\gamma}(\theta)\right)^2+4|\hat{\boldsymbol{h}}||\boldsymbol{h}|\hat{\gamma}(\theta)\left(P_{k(\theta)}(\phi,\theta)-Q(\phi,\theta)\right)\geq 0
    \end{equation}
    By the definition \eqref{eqn:def minimal stabilizing function} of $k_0(\theta)$, we have $P_{k(\theta)}(\phi,\theta)-Q(\phi,\theta)\geq 0,\,\,\forall\phi\in2\pi\mathbb{T}$ for any $k(\theta)\geq k_0(\theta)$, thus \eqref{estimate finial equation} holds true. By Theorem~\ref{thm:existence of k0}, the minimal stabilizing function $k_0(\theta)<+\infty$ is well-defined, therefore we can choose sufficiently large $k(\theta)\geq k_0(\theta)$ such that the intended local energy estimate \eqref{eqn:local energy estimate} is validated.
\end{proof}

\subsection{Proof of the main result}

Leveraging the local energy estimate \eqref{eqn:local energy estimate} as outlined in Lemma~\ref{lemma:local energy estimate}, we can now prove the unconditional energy stability aspect of the main result, as stated in Theorem~\ref{MAIN RESULT}.

\begin{proof}
    Suppose $k(\theta)$ is sufficiently large such that $k(\theta)\geq k_0(\theta)$. We take $\boldsymbol{h}=\boldsymbol{h}^m_j,\hat{\boldsymbol{h}}=\boldsymbol{h}^{m+1}_j$ in the local energy estimate \eqref{eqn:local energy estimate}: \begin{equation}
        \frac{1}{|\boldsymbol{h}^m_j|}\left(\hat{\boldsymbol{G}}_k(\theta_j^m)\boldsymbol{h}_j^{m+1}\right)\cdot(\boldsymbol{h}_j^{m+1}-\boldsymbol{h}_j^m)\geq |\boldsymbol{h}_j^{m+1}|\hat{\gamma}(\theta_j^{m+1})-|\boldsymbol{h}_j^m|\hat{\gamma}(\theta_j^m).
    \end{equation} Therefore, for any $m\geq 0$, \begin{equation}\label{eqn:energy difference}
        \begin{aligned}
            \left(\hat{\boldsymbol{G}}_k(\theta^m)\partial_s\boldsymbol{X}^{m+1},\partial_s(\boldsymbol{X}^{m+1}-\boldsymbol{X}^m)\right)_{\Gamma^m}^h&=\sum_{j=1}^N \left[|\boldsymbol{h}_j^m|\left(\hat{\boldsymbol{G}}_k(\theta_j^m)\frac{\boldsymbol{h}_j^{m+1}}{|\boldsymbol{h}_j^m|}\right)\cdot\frac{\boldsymbol{h}_j^{m+1}-\boldsymbol{h}_j^m}{|\boldsymbol{h}_j^m|}\right]\\
            &=\sum_{j=1}^N\left[\frac{1}{|\boldsymbol{h}^m_j|}\left(\hat{\boldsymbol{G}}_k(\theta_j^m)\boldsymbol{h}_j^{m+1}\right)\cdot(\boldsymbol{h}_j^{m+1}-\boldsymbol{h}_j^m)\right]\\
            &\geq\sum_{j=1}^N\left[|\boldsymbol{h}_j^{m+1}|\hat{\gamma}(\theta_j^{m+1})-|\boldsymbol{h}_j^m|\hat{\gamma}(\theta_j^m)\right]\\
            &=\sum_{j=1}^N|\boldsymbol{h}_j^{m+1}|\hat{\gamma}(\theta_j^{m+1})-\sum_{j=1}^N|\boldsymbol{h}_j^m|\hat{\gamma}(\theta_j^m)\\
            &=W_c^{m+1}-W_c^m.
        \end{aligned}
    \end{equation}
    
    Taking $\varphi^h=\mu^{m+1},\boldsymbol{\omega}^h=\boldsymbol{X}^{m+1}-\boldsymbol{X}^m$ in \eqref{eqn:sp-pfem surface diffusion}, we have \begin{equation}
        \begin{aligned}
            \Bigl(\hat{\boldsymbol{G}}_k(\theta^m)\partial_s\boldsymbol{X}^{m+1},\partial_s(\boldsymbol{X}^{m+1}-\boldsymbol{X}^m)\Bigr)^h_{\Gamma^m}&=\Bigl(\mu^{m+1}\boldsymbol{n}^{m+\frac{1}{2}},\boldsymbol{X}^{m+1}-\boldsymbol{X}^m\Bigr)^h_{\Gamma^m}\\
            &=-\tau\Bigl(\partial_s\mu^{m+1},\partial_s\mu^{m+1}\Bigr)^h_{\Gamma^m}.
        \end{aligned}
    \end{equation} Combining this with \eqref{eqn:energy difference} yields that \begin{equation}
        \begin{aligned}
            W^{m+1}_c-W^m_c&\leq \left(\hat{\boldsymbol{G}}_k(\theta^m)\partial_s\boldsymbol{X}^{m+1},\partial_s(\boldsymbol{X}^{m+1}-\boldsymbol{X}^m)\right)_{\Gamma^m}^h\\
            &\leq-\tau\left(\partial_s\mu^{m+1},\partial_s\mu^{m+1}\right)_{\Gamma^m}^h\\
            &\leq 0,\qquad\forall m\geq 0, 
        \end{aligned}
    \end{equation} which immediately implies the unconditional energy stability in \eqref{thm:main results}.
\end{proof}

\section{Existence of the minimal stabilizing function}\label{section: existence of k0}

In this section, we present a proof of the existence of the minimal stabilizing function corresponding to $\hat{\gamma}(\theta)$ that satisfies the energy stable condition in \eqref{es condition on gamma}. 

Recall the definition \eqref{eqn:auxiliary functions} of $P_\alpha,Q$, denote \begin{equation}\label{eqn:F=P^2-Q^2}
        \begin{aligned}
            F_\alpha(\phi,\theta)&\coloneqq P_\alpha^2(\phi,\theta)-Q^2(\phi,\theta)\\
            &=4\hat{\gamma}(\theta)\left(\hat{\gamma}(\theta)+\alpha\sin^2\phi\right)\\
            &\quad-\left(\hat{\gamma}(\theta-\phi)+\hat{\gamma}(\theta)\cos\phi+\hat{\gamma}^\prime(\theta)\sin\phi\right)^2.
        \end{aligned}
\end{equation} Suppose $C>0$ is a positive number such that $\frac{1}{C}\leq\hat{\gamma}(\theta)\leq C$ and $|\hat{\gamma}^\prime(\theta)|,|\hat{\gamma}^{\prime\prime}(\theta)|\leq C,\,\,\forall\theta\in2\pi\mathbb{T}$. 

Before formally commencing our proof, we first need the following two technical lemmas:

\begin{lem}\label{lemma:phi=0}
    For a globally $C^1$ and piecewise $C^2$ anisotropy $\hat{\gamma}(\theta)$, there exists an
    open neighborhood $U_0$ of $0$ and a positive constant $k_{U_0}<+\infty$ such that for any $\alpha> k_{U_0}$, we have  $F_\alpha(\phi,\theta)\geq 0,\forall\phi\in U_{0}$.
\end{lem}

\begin{proof}
     Since $\hat{\gamma}(\theta)$ is globally $C^1$ and piecewise $C^2$, there exists a $0<\varepsilon<\frac{\pi}{2}$ such that $\hat{\gamma}(\theta)$ is $2$-times continuously differentiable on both $[\theta-\varepsilon,\theta]$ and $[\theta,\theta+\varepsilon]$. 
     
     Applying the mean value theorem to $[\theta-\varepsilon,\theta]$ and $[\theta,\theta+\varepsilon]$, we deduce that for $\phi\in U_0 \coloneqq\left\{\phi\mid |\phi|<\varepsilon\right\}$, there exists a $\xi$ between $\theta$ and $\theta-\phi$ such that \begin{equation}\label{eqn taylor gamma}
         \hat{\gamma}(\theta-\phi)=\hat{\gamma}(\theta)-\hat{\gamma}^\prime(\theta)\phi+\frac{\hat{\gamma}^{\prime\prime}(\xi)}{2}\phi^2.
     \end{equation}
     Again, by the mean value theorem, there exist $\xi_1,\xi_2$ between $0,\phi$ such that \begin{subequations}\label{eqn:the taylor expansions}
         \begin{align}
             \cos\phi&=1-\frac{1}{2}\phi^2+\frac{\sin\xi_1}{6}\phi^3\\
             \sin\phi&=\phi-\frac{\cos\xi_2}{6}\phi^3.
         \end{align}
     \end{subequations} Substituting \eqref{eqn taylor gamma} and \eqref{eqn:the taylor expansions} into $2\hat{\gamma}(\theta)-Q(\phi,\theta)$, we have \begin{equation}
         \begin{aligned}
             2\hat{\gamma}(\theta)-Q(\phi,\theta)&=2\hat{\gamma}(\theta)-\hat{\gamma}(\theta)+\hat{\gamma}^\prime(\theta)\phi-\frac{1}{2}\hat{\gamma}^{\prime\prime}(\xi)\phi^2\\
             &\quad-\hat{\gamma}(\theta)+\frac{1}{2}\hat{\gamma}(\theta)\phi^2-\frac{1}{6}\hat{\gamma}(\theta)\sin\xi_1\phi^3\\
             &\quad-\hat{\gamma}^\prime(\theta)\phi+\frac{1}{6}\hat{\gamma}^\prime(\theta)\cos\xi_2\phi^3\\
             &=\frac{1}{2}\left(\hat{\gamma}(\theta)-\hat{\gamma}^{\prime\prime}(\xi)\right)\phi^2-\frac{1}{6}\left(\hat{\gamma}(\theta)\sin\xi_1-\hat{\gamma}^\prime(\theta)\cos\xi_2\right)\phi^3.
         \end{aligned}
     \end{equation} Thus, \begin{equation}
         \begin{aligned}
             |2\hat{\gamma}(\theta)-Q(\phi,\theta)|&\leq\frac{1}{2}\left(|\hat{\gamma}(\theta)|+|\hat{\gamma}^{\prime\prime}(\xi)|\right)|\phi|^2 \\
             &\quad+\frac{1}{6}\left(|\hat{\gamma}(\theta)\sin\xi_1|+|\hat{\gamma}^\prime(\theta)\cos\xi_2|\right)|\phi|^3\\
             &\leq C|\phi|^2+\frac{C}{3}|\phi|^3.
         \end{aligned}
     \end{equation} 
     Since $|2\hat{\gamma}(\theta)+Q(\phi,\theta)|\leq 2|\hat{\gamma}(\theta)|+|Q(\phi,\theta)|\leq 5C$, then for any $\phi\in U_0$, \begin{equation}
         \begin{aligned}
             |4\hat{\gamma}^2(\theta)-Q^2(\phi,\theta)|&\leq|2\hat{\gamma}(\theta)+Q(\phi,\theta)||2\hat{\gamma}(\theta)-Q(\phi,\theta)|\\
             &\leq 5C\left(C|\phi|^2+\frac{C}{3}|\phi|^3\right)\\
             &=\frac{5C^2}{3}\left(3+|\phi|\right)|\phi|^2\\
             &\leq \frac{5(6+\pi)C^2}{6}|\phi|^2.
         \end{aligned}
     \end{equation} 

     Noting that for any $\phi\in U_0$, we have $|\sin\phi|>\frac{2}{\pi}|\phi|$, therefore, \begin{equation}
         \begin{aligned}
             F_\alpha(\phi,\theta)&=4\hat{\gamma}(\theta)\alpha\sin^2\phi+4\hat{\gamma}^2(\theta)-Q^2(\phi,\theta)\\
             &\geq \frac{16}{C\pi^2}\alpha |\phi|^2-\frac{5(6+\pi)C^2}{6}|\phi|^2\\
             &\geq\left[\frac{16}{C\pi^2}\alpha-\frac{5(6+\pi)C^2}{6}\right]|\phi|^2,\,\,\forall\phi\in U_{0}.
         \end{aligned}
     \end{equation} Thus there exists a positive constant $k_{U_0}\coloneqq\frac{5\pi^2(6+\pi)C^3}{96}$, for $\alpha>k_{U_0}$, $F_{\alpha}(\phi,\theta)\geq 0,\forall\phi\in U_0$.
\end{proof}

\begin{lem}\label{lemma:phi=pi}
    Suppose $\hat{\gamma}(\theta)$ satisfying \eqref{es condition on gamma}. There exists an open neighborhood $U_{\pi}$ of $\pi$ with a positive $k_{U_\pi}<+\infty$ such that for any $\alpha> k_{U_\pi}$, $F_\alpha(\phi,\theta)\geq 0,\forall\phi\in U_{\pi}$.
\end{lem}

\begin{proof}
    \begin{enumerate}[(i)]
        \item If $3\hat{\gamma}(\theta)>\hat{\gamma}(\theta-\pi)$ for any $\theta\in 2\pi\mathbb{T}$, then \begin{equation}
            F_{\alpha}(\pi,\theta)=\left(3\hat{\gamma}(\theta)-\hat{\gamma}(\theta-\pi)\right)\left(\hat{\gamma}(\theta)+\hat{\gamma}(\theta-\pi)\right)>0.
        \end{equation} Thus, by the continuity of $F_\alpha(\phi,\theta)$, there exists an open neighborhood $U_\pi$ of $\pi$ such that $F_{\alpha}(\pi,\theta)\geq 0,\forall\phi\in U_{\pi}$.
        \item If $3\hat{\gamma}(\theta^*)=\hat{\gamma}(\theta^*-\pi)$ and $\hat{\gamma}^\prime(\theta^*)=0$ at $\theta^*\in2\pi\mathbb{T}$. Since function $\gamma(\theta-\pi)/\gamma(\theta)\leq 3$, it attends its maximum at $\theta^*$, thus \begin{equation}
            \begin{aligned}
                \left.\left(\frac{\hat{\gamma}(\theta-\pi)}{\hat{\gamma}(\theta)}\right)^\prime\right|_{\theta=\theta^*}&=\frac{\hat{\gamma}^\prime(\theta^*-\pi)\hat{\gamma}(\theta^*)-\hat{\gamma}(\theta^*-\pi)\hat{\gamma}^\prime(\theta^*)}{\hat{\gamma}^2(\theta^*)}\\
                &=\frac{\hat{\gamma}^\prime(\theta^*-\pi)-3\hat{\gamma}^\prime(\theta^*)}{\hat{\gamma}(\theta^*)}\\
                &=\frac{\hat{\gamma}^\prime(\theta^*-\pi)}{\hat{\gamma}(\theta^*)}=0,
            \end{aligned}
        \end{equation} i.e., $\hat{\gamma}^\prime(\theta^*-\pi)=0$. 
        
        Similar to derivations of \eqref{eqn taylor gamma}, according to the mean value theorem, we know that there exists positive number $0<\varepsilon<\frac{\pi}{2}$ and a neighborhood $U_\pi=\{\phi\mid|\phi-\pi|<\varepsilon\}$. Such that for any $\phi\in U_\pi$, there exists a $\xi$ between $\theta^*-\phi$ and $\theta^*-\pi$ satisfying \begin{equation}\label{eqn:taylor of gamma2}
            \begin{aligned}
                \hat{\gamma}(\theta^*-\phi)&=\hat{\gamma}(\theta^*-\pi)-\hat{\gamma}^\prime(\theta^*-\pi)(\phi-\pi)+\frac{\hat{\gamma}^{\prime\prime}(\xi)}{2}(\phi-\pi)^2\\
                &=\hat{\gamma}(\theta^*-\pi)+\frac{\hat{\gamma}^{\prime\prime}(\xi)}{2}(\phi-\pi)^2.
            \end{aligned}
        \end{equation}
    
    Again, by the mean value theorem, we know that there exists a $\xi_1$ between $\pi,\phi$, \begin{equation}\label{eqn:taylor of cos2}
            \cos\phi=-1-\frac{\cos\xi_1}{6}(\phi-\pi)^2.
        \end{equation}
    By substituting \eqref{eqn:taylor of gamma2} and \eqref{eqn:taylor of cos2}  into $2\hat{\gamma}(\theta^*)-Q(\phi,\theta^*)$, we have \begin{equation}
        \begin{aligned}
            2\hat{\gamma}(\theta^*)-Q(\phi,\theta^*)&=2\hat{\gamma}(\theta^*)-\hat{\gamma}^\prime(\theta^*-\pi)-\frac{\hat{\gamma}^{\prime\prime}(\xi)}{2}(\phi-\pi)^2\\
            &\quad +\hat{\gamma}(\theta^*)+\frac{\cos\xi_1}{6}\hat{\gamma}(\theta)(\phi-\pi)^2+\hat{\gamma}(\theta^*)\sin\phi\\
            &=\frac{1}{6}\left(\hat{\gamma}(\theta^*)\cos\xi_1-3\hat{\gamma}^{\prime\prime}(\xi)\right)(\phi-\pi)^2.
        \end{aligned}
    \end{equation} Thus, \begin{equation}
        \begin{aligned}
            |2\hat{\gamma}(\theta^*)-Q(\phi,\theta^*)|&\leq \frac{1}{6}\left(|\hat{\gamma}(\theta^*)\cos\xi_1|+3|\hat{\gamma}^{\prime\prime}(\xi)|\right)|\phi-\pi|^2\\
            &\leq \frac{2C}{3}|\phi-\pi|^2.
        \end{aligned}
    \end{equation}

    Given that $|2\hat{\gamma}(\theta^*)+Q(\phi,\theta^*)|\leq 5C$, we can deduce that \begin{equation}
        \begin{aligned}
            |4\hat{\gamma}^2(\theta^*)-Q^2(\phi,\theta^*)|&\leq |2\hat{\gamma}(\theta^*)+Q(\phi,\theta^*)||2\hat{\gamma}(\theta^*)-Q(\phi,\theta^*)|\\
            &\leq \frac{10C^2}{3}|\phi-\pi|^2.
        \end{aligned}
    \end{equation} Noting that for any $\phi\in U_{\pi}$, we have $|\sin\phi|>\frac{2}{\pi}|\phi-\pi|$ since $\varepsilon\in(0,\pi)$. Therefore, \begin{equation}
        \begin{aligned}
            F_\alpha(\phi,\theta^*)&=4\hat{\gamma}(\theta^*)\alpha\sin^2\phi+4\hat{\gamma}^2(\theta^*)-Q^2(\phi,\theta^*)\\
            &\geq \frac{16}{C\pi^2}\alpha|\phi-\pi|^2-\frac{10C^2}{3}|\phi-\pi|^2\\
            &\geq\left[\frac{16}{C\pi^2}\alpha-\frac{10C^2}{3}\right]|\phi-\pi|^2,\,\,\forall\phi\in U_{\pi}.
        \end{aligned}
    \end{equation} Therefore, there exists a positive constant $k_{U_\pi}\coloneqq\frac{5\pi^2C^3}{24}$, such that for $\alpha>k_{U_\pi}$, we have $F_{\alpha}(\phi,\theta)\geq 0,\forall\phi\in U_\pi$.
     \end{enumerate}
\end{proof}

With the help of above two lemmas, Theorem~\ref{thm:existence of k0} can now be proven:

\begin{proof} (\textit{Existence of the minimal stabilizing function})
    \begin{enumerate}[(i)]
        \item For any $\sin\phi_0\neq 0$, i.e. $\phi_0\neq 0,\pi$, there exists an open neighborhood $U_{\phi_0}$ of $\phi_0$ such that $\sin^2\phi$ has a strict positive lower bound in $U_{\phi_0}$, i.e. there exists a constant $c>0$ such that $\sin^2\phi\geq c>0,\forall \phi\in U_{\phi_0}$, then we have \begin{equation}
        \begin{aligned}
            F_\alpha(\phi,\theta)&=4\hat{\gamma}(\theta)\alpha\sin^2\phi+O(1)\\
            &\geq 4c\hat{\gamma}(\theta)\alpha+O(1).
        \end{aligned}
        \end{equation} Thus there positive exists a constant $k_{U_{\phi_0}}<+\infty$, for any $\alpha>k_{U_{\phi_0}}$, $F_\alpha(\phi,\theta)\geq 0,\forall\phi\in U_{\phi_0}$.\\
        \item For $\phi_0=0$, by Lemma~\ref{lemma:phi=0}, there exists an open neighborhood $U_{\phi_0}\coloneqq U_0$ of $0$ and a positive constant $k_{U_{\phi_0}}\coloneqq k_{U_0}<+\infty$ such that for any $\alpha>k_{U_{\phi_0}}$, $F_\alpha(\phi,\theta)\geq 0,\forall\phi\in U_{\phi_0}$.\\
        \item For $\phi_0=\pi$, by Lemma~\ref{lemma:phi=pi}, there also exists an open neighborhood $U_{\phi_0} \coloneqq U_\pi$ of $\pi$ with a positive constant $k_{U_{\phi_0}} \coloneqq k_{U_\pi}<+\infty$ such that for any $\alpha>k_{U_{\phi_0}}$, $F_\alpha(\phi,\theta)\geq 0,\forall\phi\in U_{\phi_0}$.
    \end{enumerate}

Since $2\pi\mathbb{T}$ is compact and $\{U_{\phi_0}:\phi_0\in 2\pi\mathbb{T}\}$
forms an open cover of $2\pi\mathbb{T}$, by the open cover theorem, one can select a finite subcover $\{U_i\}_{i=1}^{M}\subset\{U_{\phi_0}:\phi_0\in 2\pi\mathbb{T}\}$, then for $\displaystyle\alpha>\max_{1\leq i\leq M}k_{U_i}$, we have \begin{equation}
    F_\alpha(\phi,\theta)\geq 0,\qquad\forall\phi\in 2\pi\mathbb{T}.
\end{equation} Which implies $k_0(\theta)=\inf\left\{\alpha\geq 0\mid P_\alpha(\phi,\theta)-Q(\phi,\theta)\geq 0,\,\,\forall\phi\in2\pi\mathbb{T}\right\}<+\infty$.
\end{proof}

\section{Extension to solid-state dewetting}\label{section solid-state dewetting}

In this section, we extend the conservative formulation \eqref{eqn:surface diffusion reformulate conservative form} and its SPFEM \eqref{eqn:sp-pfem surface diffusion} for a closed curve under anisotropic surface diffusion to solid-state dewetting in materials science \cite{bao2017parametric,jiang2016solid,wang2015sharp}.

\subsection{Sharp interface model and a stabilized formulation}

\begin{figure}[t!]
\centering
\includegraphics[width=0.5\textwidth]{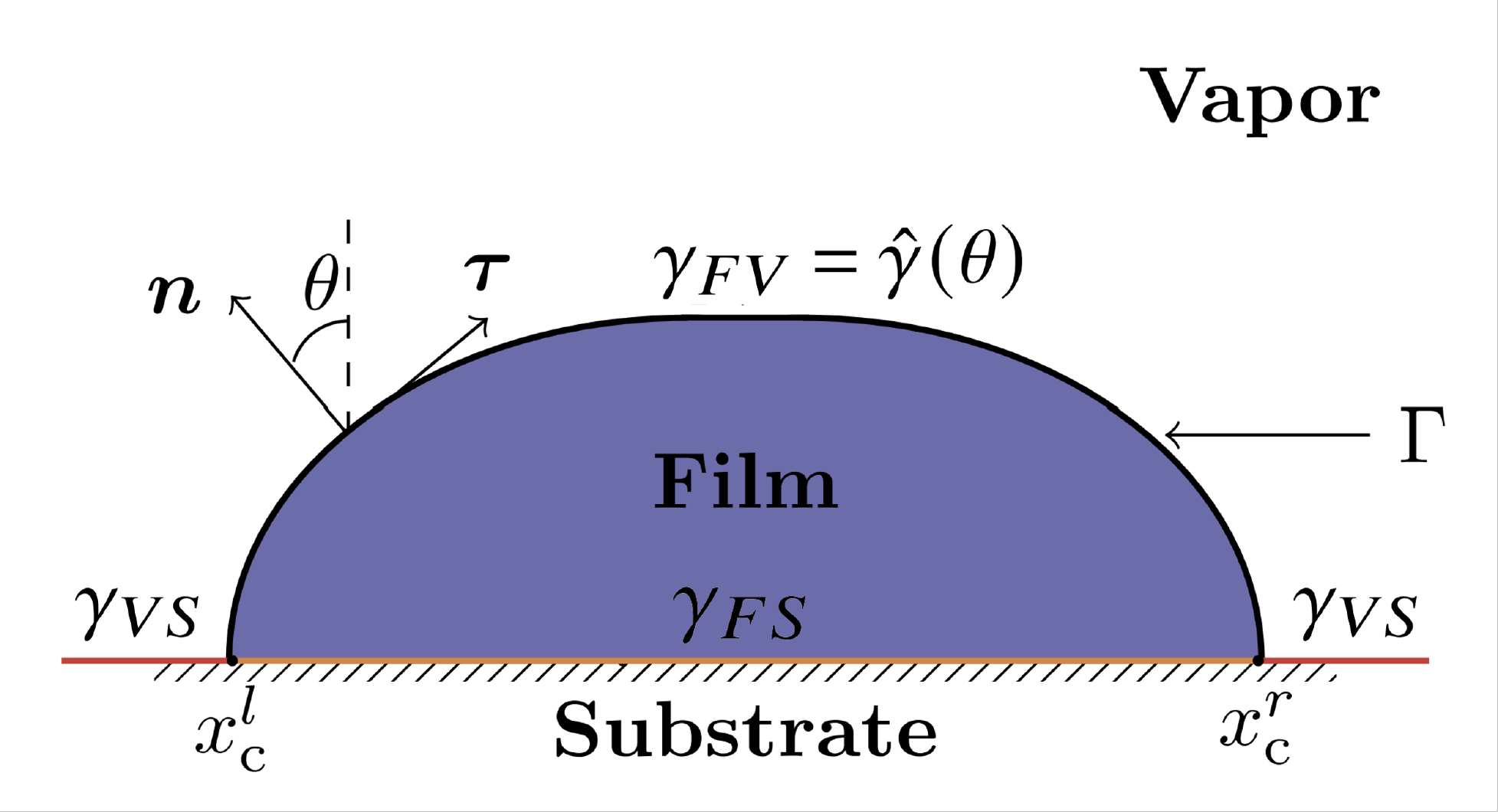}
\caption{An illustration of solid-state dewetting of a thin film on a flat rigid substrate in 2D, where $\gamma_{FV}=\hat{\gamma}(\theta),\gamma_{VS},\gamma_{FS}$ represent surface energy densities of film/vapor, vapor/substrate and film/substrate interface, respectively, $x_c^l,x_c^r$ are the left and right contact points.}
\label{fig: illustration of solid-state dewetting}
\end{figure}

As shown in Fig.~\ref{fig: illustration of solid-state dewetting}, the solid-state dewetting problem in 2D is described as evolution of an open curve $\Gamma(t)=\boldsymbol{X}(s,t)=(x(s,t),y(s,t))^T\,\,(0\leq s\leq L)$ under anisotropic surface diffusion and contact line migration. Here, $s$ and $t$ are the arc-length parameter and time, respectively, and $L\coloneqq L(t)=|\Gamma(t)|$ represents the total length of $\Gamma(t)$. As it was derived in the literature \cite{bao2017parametric,jiang2016solid,wang2015sharp}, a dimensionless sharp-interface model for simulating solid-state dewetting of thin films with weakly anisotropic surface energy can be formulated as: $\boldsymbol{X}(s,t)$ satisfying anisotropic surface diffusion \eqref{eqn:surface diffusion}--\eqref{eqn:weighted curvature} with following boundary conditions \begin{enumerate}
    \item contact point condition \begin{equation}\label{bdc contact points}
        y(0,t)=0,\qquad y(L,t)=0,\qquad t\geq 0;
    \end{equation} 
    \item relaxed contact angle condition \begin{equation}\label{bdc relaxed contact angle}
        \frac{\mathrm{d}x_c^l(t)}{\mathrm{d}t}=\eta f(\theta^l_d;\sigma),\qquad \frac{\mathrm{d}x_c^r(t)}{\mathrm{d}t}=-\eta f(\theta^r_d;\sigma),\qquad t\geq 0;  \end{equation}
    \item zero-mass flux condition \begin{equation}\label{bdc zero-mass flux}
        \partial_s\mu(0,t)=0,\qquad \partial_s\mu(L,t)=0,\qquad t\geq 0;
    \end{equation} 
\end{enumerate} where $x_c^l(t)=x(0,t)\leq x_c^r(t)=x(L,t)$ are the left and right contact points, $\theta_d^l\coloneqq\theta_d^l(t),\theta_d^r\coloneqq \theta_d^r(t)$ are the contact angles at each contact points, respectively. $f(\theta;\sigma)$ is defined as \begin{equation}
    f(\theta;\sigma)\coloneqq\hat{\gamma}(\theta)\cos\theta-\hat{\gamma}^\prime(\theta)\sin\theta-\sigma,\qquad\theta\in 2\pi\mathbb{T}.
\end{equation} with $\sigma\coloneqq\cos\theta_Y=\frac{\gamma_{VS}-\gamma_{FS}}{\gamma_{FV}}$ and $\theta_Y$ be the isotropic Young contact angle. $0<\eta<+\infty$ denotes the contact line mobility \cite{bao2017parametric,jiang2016solid,wang2015sharp}.

Similarly to \eqref{eqn:surface diffusion reformulate conservative form}, a stabilized formulation for solid-state dewetting can be derived as: \begin{subequations}\label{eqn:stab solid-state dewetting reformulate}
    \begin{align}
        &\boldsymbol{n}\cdot\partial_t\boldsymbol{X}-\partial_{ss}\mu=0,\qquad 0<s<L(t),\qquad t>0,\label{eqn:stab solid-state dewetting reformluate a}\\
        &\mu\boldsymbol{n}+\partial_s\Bigl(\hat{\boldsymbol{G}}_k(\theta)\partial_{s}\boldsymbol{X}\Bigr)=\boldsymbol{0},\label{eqn:stab solid-state dewetting reformluate b}
    \end{align}
\end{subequations} with boundary conditions \eqref{bdc contact points}--\eqref{bdc zero-mass flux}, where $\boldsymbol{\hat{G}}_k(\theta)$ is the stabilizing surface energy matrix given in \eqref{surface energy matrix}.

\subsection{A stabilized weak formulation}

Suppose $\Gamma(t)$ is parameterized by a time-independent variable $\rho$ over a fixed domain $\rho\in\mathbb{I}=[0,1]$, i.e. \begin{equation}
    \Gamma(t)\coloneqq\boldsymbol{X}(\rho,t)=(x(\rho,t),y(\rho,t))^T\colon\mathbb{I}\times\mathbb{R}^+\to\mathbb{R}^2,
\end{equation} thus the arc-length parameter $s$ can be given as $s=\int_0^\rho|\partial_q\boldsymbol{X}(q,t)|\,\mathrm{d}q$. (Again, we do not distinguish $\boldsymbol{X}(s,t)$ and $\boldsymbol{X}(\rho,t)$.)

Introducing the following functional spaces as \begin{equation}
    H^1_0(\mathbb{I})\coloneqq\left\{u\in H^1(\mathbb{I})\mid u(0)=u(1)=0\right\},\qquad\mathbb{X}\coloneqq H^1(\mathbb{I})\times H^1_0(\mathbb{I}).
\end{equation} Following similar derivations as in Section \ref{section sp-pfem}, we present the new variational formulation for \eqref{eqn:stab solid-state dewetting reformulate} with boundary conditions \eqref{bdc contact points}--\eqref{bdc zero-mass flux} as follows: Given an initial open curve $\Gamma(0)=\boldsymbol{X}(\cdot,0)=\boldsymbol{X}_0\in\mathbb{X}$, find the solution $(\boldsymbol{X}(\cdot,t),\mu(\cdot,t))\in\mathbb{X}\times H^1(\mathbb{I})$ such that \begin{subequations}\label{new variational formulation solid-state dewetting}
    \begin{align}
        \begin{split}
            &\Bigl(\boldsymbol{n}\cdot \partial_t\boldsymbol{X}, \varphi\Bigr)_{\Gamma(t)}+\Bigl(\partial_s\mu, \partial_s\varphi\Bigr)_{\Gamma(t)}=0,\qquad \forall\varphi\in H^1(\mathbb{I}),
        \end{split}\\
        \begin{split}
            &\Bigl(\mu\boldsymbol{n},\boldsymbol{\omega}\Bigr)_{\Gamma(t)}-\Bigl(\hat{\boldsymbol{G}}_k(\theta)\partial_s\boldsymbol{X},\partial_s\boldsymbol{\omega}\Bigr)_{\Gamma(t)}-\frac{1}{\eta}\left[\frac{\mathrm{d}x_c^l(t)}{\mathrm{d}t}\omega_1(0)+\frac{\mathrm{d}x_c^r(t)}{\mathrm{d}t}\omega_1(1)\right]\\
            &\qquad\qquad\,\,\, +\sigma\left[\omega_1(1)-\omega_1(0)\right]=0,\qquad\forall\boldsymbol{\omega}=(\omega_1,\omega_2)^T\in\mathbb{X}.
        \end{split}
    \end{align}
\end{subequations}

The area $A_o(t)$ enclosed by $\Gamma(t)$ and the substrate and the total energy $W_o(t)$ are defined as \begin{equation}
    A_o(t)\coloneqq \int_{\Gamma(t)}y(s,t)\partial_sx(s,t)\,\mathrm{d}s,\quad W_o(t)\coloneqq \int_{\Gamma(t)}\hat{\gamma}(\theta)\,\mathrm{d}s-\sigma\left(x_c^r(t)-x_c^l(t)\right).
\end{equation} 
Similar to the closed curve case, the weak formulation \eqref{new variational formulation solid-state dewetting} still preserves area conservation and energy dissipation.

\begin{prop}[area conservation and energy dissipation] Suppose $\Gamma(t)$ is given by the solution $(\boldsymbol{X}(\cdot, t), \mu(\cdot, t))$ of the weak formulation \eqref{new variational formulation solid-state dewetting}. Then, \begin{equation}\label{geo properties, cont, sdeweting}
    A_o(t)\equiv A_o(0),\qquad W_o(t)\leq W_o(t_1)\leq W_o(0), \qquad \forall t\geq t_1\geq 0.
\end{equation} More precisely, \begin{equation}\label{eqn:area conservation & energy disspation of solid-state dewetting}
    \frac{\mathrm{d}}{\mathrm{d}t}A_o(t)=0,\quad \frac{\mathrm{d}}{\mathrm{d}t}W_o(t)=-\int_{\Gamma(t)}|\partial_s\mu|^2\,\mathrm{d}s-\frac{1}{\eta}\left[\left(\frac{\mathrm{d}x_c^r}{\mathrm{d}t}\right)^2+\left(\frac{\mathrm{d}x_c^l}{\mathrm{d}t}\right)^2\right]\leq 0,\quad t\geq 0.
\end{equation}
\end{prop}

The proof is similar to Proposition \ref{prop weak form structure preserving}, we omit the details here for brevity.

\subsection{A SPFEM and its proprieties}

Introducing the finite element spaces \begin{equation}
    \mathbb{K}_0^h\coloneqq \left\{u^h\in\mathbb{K}^h\mid u^h(0)=u^h(1)=0\right\},\qquad\mathbb{X}^h\coloneqq \mathbb{K}^h\times\mathbb{K}^h_0.
\end{equation} The semi-discretization in space of \eqref{new variational formulation solid-state dewetting} and its area conservation and energy dissipation properties are similar to Section \ref{section sp-pfem}, we omit the details here for brevity.

Let $\tau$ be the uniform time step, denotes the approximation of $\Gamma(t)=\boldsymbol{X}(\cdot,t)$ at $t_m=m\tau,m=0,1,\dots$ as $\Gamma^m=\boldsymbol{X}^m(\cdot)=\cup_{j=1}^N\boldsymbol{h}_j^m$ where $\boldsymbol{h}_j^m\coloneqq \boldsymbol{X}^m(\rho_j)-\boldsymbol{X}^m(\rho_{j-1})$. Then the structure-preserving discretization of SPFEM for solid-state dewetting \eqref{eqn:stab solid-state dewetting reformulate} can be stated as: Suppose the initial approximation $\Gamma^0(\cdot)\in \mathbb{X}^h$ is given by $\boldsymbol{X}^0(\rho_j)=\boldsymbol{X}_0(\rho_j), \forall 0\leq j\leq N$. For any $m=0, 1, 2, \ldots$, find the solution $(\boldsymbol{X}^{m+1}(\cdot)=(x^{m+1}(\cdot),y^{m+1}(\cdot))^T, \mu^{m+1}(\cdot))\in \mathbb{X}^h\times \mathbb{K}^h$, such that
\begin{subequations}\label{eqn:sp-pfem solid-state dewetting}
    \begin{align}
        \begin{split}
            &\Bigl(\boldsymbol{n}^{m+\frac{1}{2}}\cdot\frac{\boldsymbol{X}^{m+1}-\boldsymbol{X}^m}{\tau}, \varphi^{h}\Bigr)_{\Gamma^m}^h+\Bigl(\partial_s\mu^{m+1}, \partial_s\varphi^h\Bigr)_{\Gamma^m}^h=0,\qquad \forall\varphi^h\in \mathbb{K}^h,
        \end{split}\\
        \begin{split}
            &\Bigl(\mu^{m+1}\boldsymbol{n}^{m+\frac{1}{2}},\boldsymbol{\omega}^h\Bigr)_{\Gamma^m}^h-\Bigl(\hat{\boldsymbol{G}}_k(\theta^m)\partial_s\boldsymbol{X}^m,\partial_s\boldsymbol{\omega}^h\Bigr)_{\Gamma^m}^h-\frac{1}{\eta}\left[\frac{x_l^{m+1}-x_l^m}{\tau}\omega_1^h(0)+\frac{x_r^{m+1}-x_r^m}{\tau}\omega_1^h(1)\right] \\
            &\qquad\qquad\qquad\,\,\,+\sigma\left[\omega_1^h(1)-\omega_1^h(0)\right]=0,\qquad\forall\boldsymbol{\omega}^h=(\omega_1^h,\omega_2^h)\in\mathbb{X}^h,
        \end{split}
    \end{align}
\end{subequations} satisfying $x_l^{m+1}=x^{m+1}(0)\leq x_r^{m+1}=x^{m+1}(1)$. The definition of $\boldsymbol{G}_k(\theta^m)$ is similar to \eqref{dis surface energy matrix}.

Denote the area $A_o^m$ enclosed by the open polygonal curve $\Gamma(t)$ and the substrate and the total surface energy $W_o^m$ as \begin{subequations}
    \begin{align}
        A_o^m&\coloneqq \frac{1}{2}\sum_{j=1}^N\left(x^m(\rho_j)-x^m(\rho_{j-1})\right)\left(y^m(\rho_j)+y^m(\rho_{j-1})\right),\\
        W_o^m&\coloneqq \sum_{j=1}^N|\boldsymbol{h}_j^m|\hat{\gamma}(\theta^m_j)-\sigma\left(x_r^m-x_l^m\right).
    \end{align}
\end{subequations} Then for the SPFEM \eqref{eqn:sp-pfem solid-state dewetting}, we have following result on its structure preservation.

\begin{thm}[structure-preserving]
    For any $\hat{\gamma}(\theta)$ satisfying \eqref{es condition on gamma}, the SPFEM \eqref{eqn:sp-pfem solid-state dewetting} is area conservative and unconditional energy dissipative with sufficiently large $k(\theta)$, i.e. \begin{equation}\label{eqn:structure preserving open}
        A_o^{m+1}=A_o^m=\cdots=A_o^0,\qquad W_o^{m+1}\leq W_o^m\leq\cdots\leq W_o^0,\qquad\forall m\geq 0.
    \end{equation}
\end{thm}

\begin{proof}
    For the area conservation part, the proof is similar to \cite[Theorem~2.1]{bao2021structure} which is omitted here for brevity.
  
    For the energy dissipation part, take $\varphi^h=\mu^{m+1},\boldsymbol{\omega}^h=\boldsymbol{X}^{m+1}-\boldsymbol{X}^m$ in \eqref{eqn:sp-pfem solid-state dewetting}, then \begin{equation}\label{energy lemma open}
        \begin{aligned}
            \tau&\left(\partial_s\mu^{m+1},\partial_s\mu^{m+1}\right)_{\Gamma^m}^h+\left(\hat{\boldsymbol{G}}_k(\theta^m)\partial_s\boldsymbol{X}^{m+1},\partial_s(\boldsymbol{X}^{m+1}-\boldsymbol{X}^m)\right)_{\Gamma^m}^h\\
            &\qquad+\frac{1}{\eta}\left[\frac{(x_l^{m+1}-x_l^m)^2}{\tau}+\frac{(x_r^{m+1}-x_r^m)^2}{\tau}\right]-\sigma\left[(x_r^{m+1}-x_r^m)-(x_l^{m+1}-x_l^m)\right]=0.
        \end{aligned}
    \end{equation} 

    Noting that Lemma~\ref{lemma:local energy estimate} is still valid for open polygonal curve under the condition \eqref{es condition on gamma} on $\hat{\gamma}(\theta)$, thus by similar derivation in \eqref{eqn:energy difference}, we have \begin{equation}\label{eqn:energy difference open}
        \left(\hat{\boldsymbol{G}}_k(\theta^m)\partial_s\boldsymbol{X}^{m+1},\partial_s(\boldsymbol{X}^{m+1}-\boldsymbol{X}^m)\right)_{\Gamma^m}^h\geq \sum_{j=1}^N|\boldsymbol{h}_j^{m+1}|\hat{\gamma}(\theta^{m+1}_j)-\sum_{j=1}^N|\boldsymbol{h}_j^{m}|\hat{\gamma}(\theta^{m}_j).
    \end{equation} Therefore, by \eqref{energy lemma open} and \eqref{eqn:energy difference open},\begin{equation}
        \begin{aligned}
            W^{m+1}_o-W^m_o&=\sum_{j=1}^N|\boldsymbol{h}_j^{m+1}|\hat{\gamma}(\theta^{m+1}_j)-\sum_{j=1}^N|\boldsymbol{h}_j^{m}|\hat{\gamma}(\theta^{m}_j)\\
            &\quad-\sigma\left[\left(x_r^{m+1}-x_r^m\right)-\left(x_l^{m+1}-x_l^m\right)\right]\\
            &\leq \left(\hat{\boldsymbol{G}}_k(\theta^m)\partial_s\boldsymbol{X}^{m+1},\partial_s(\boldsymbol{X}^{m+1}-\boldsymbol{X}^m)\right)_{\Gamma^m}^h\\
            &\quad-\sigma\left[\left(x_r^{m+1}-x_r^m\right)-\left(x_l^{m+1}-x_l^m\right)\right]\\
            &\leq-\tau\left(\partial_s\mu^{m+1},\partial_s\mu^{m+1}\right)_{\Gamma^m}^h\\
            &\quad-\frac{1}{\eta}\left[\frac{(x_l^{m+1}-x_l^m)^2}{\tau}+\frac{(x_r^{m+1}-x_r^m)^2}{\tau}\right]\\
            &\leq 0,\qquad\forall m\geq 0, 
        \end{aligned} 
    \end{equation} which implies the energy dissipation in \eqref{eqn:structure preserving open}.
\end{proof}

\begin{remark}
    Due to the local energy estimate \eqref{eqn:local energy estimate} being only dependent on $\hat{\gamma}(\theta)$, all results concerning the energy dissipation of the SPFEM \eqref{eqn:sp-pfem surface diffusion} on evolutions of closed curves can be extended to the SPFEM \eqref{eqn:sp-pfem solid-state dewetting} on evolutions of open curves in solid-state dewetting, including Remark~\ref{kfold remark}--\ref{symmetric remark}. 
\end{remark}

\section{Numerical results}\label{section numerical results}

In this section, we report numerical experiments for the proposed SPFEM \eqref{eqn:sp-pfem surface diffusion} and \eqref{eqn:sp-pfem solid-state dewetting} for time evolutions of closed curves and open curves, respectively. Extensive results are provided to illustrate their efficiency, accuracy, area conservation and unconditional energy stability. 

In the numerical tests, following typical anisotropic surface energies are considered in the simulations: \begin{itemize}
    \item Case I: $\hat{\gamma}(\theta)=1+\beta\cos 3\theta$ with $|\beta|<1$. It is weakly anisotropic when $|\beta|<\frac{1}{8}$ and strongly anisotropic otherwise;\\
    \item Case II: $\hat{\gamma}(\theta)=\sqrt{1+b\cos^2\theta}$ with $b>-1$.
\end{itemize}

To compute the minimal stabilizing function $k_0(\theta)$, we solve the optimization problem \eqref{eqn:def minimal stabilizing function} for 20 uniformly distributed points in $[-\pi,\pi]$ and do linear
interpolation for the intermediate points.

To verify the quadratic convergence rate in space and linear convergence rate in time, the time step $\tau$ is always chosen as $\tau=16h^2$ except it is stated otherwise. The manifold distance \cite{zhao2021energy,li2021energy} \begin{equation}
    M(\Gamma_1,\Gamma_2)\coloneqq 2|\Omega_1\cup\Omega_2|-|\Omega_1|-|\Omega_2|,
\end{equation} is employed to measure the distance between two closed curves $\Gamma_1,\Gamma_2$, where $\Omega_1,\Omega_2$ are the regions enclosed by $\Gamma_1,\Gamma_2$ and $|\Omega|$ represents the area of $|\Omega|$. Suppose $\Gamma^m$ is the numerical approximation of $\Gamma^h(t=t_m\coloneqq m\tau)$, thus the numerical error is defined as \begin{equation}
    e^h(t)\Big|_{t=t_m}\coloneqq M(\Gamma^m,\Gamma(t=t_m)).
\end{equation} In the Newton's iteration, the tolerance value is set to be $\text{tol}=10^{-12}$.

To test the mesh quality, the energy stability and area conservation numerically, we introduce the following indicators: the weighted mesh ratio \begin{equation}
       R_{\gamma}^h(t)\coloneqq \frac{\max\limits_{1\leq j\leq N}\hat{\gamma}(\theta_j)|\boldsymbol{h}_j|}{\min\limits_{1\leq j\leq N}\hat{\gamma}(\theta_j)|\boldsymbol{h}_j|},
\end{equation} the normalized area loss and the normalized energy for closed curves: \begin{equation}
    \left.\frac{\Delta A^h_c(t)}{A^h_c(0)}\right|_{t=t_m}\coloneqq \frac{A^m_c-A^0_c}{A_c^0},\qquad\left.\frac{W^h_c(t)}{W^h_c(0)}\right|_{t=t_m}\coloneqq \frac{W^m_c}{W^0_c},
\end{equation} and for open curves: \begin{equation}
    \left.\frac{\Delta A^h_o(t)}{A^h_o(0)}\right|_{t=t_m}\coloneqq \frac{A^m_o-A^0_o}{A_o^0},\qquad\left.\frac{W^h_o(t)}{W^h_o(0)}\right|_{t=t_m}\coloneqq \frac{W^m_o}{W^0_o}.
\end{equation}

In the following numerical tests, the initial shapes are chosen as a complete and a half ellipse with major axis $4$ and minor axis $1$ for closed curves and open curves, respectively, unless stated otherwise. The exact solution $\Gamma(t)$ is approximated by choosing $k(\theta)=k_0(\theta)$ with a small mesh size $h_e=2^{-8}$ and a time step $\tau_e=2^{-12}$ in \eqref{eqn:sp-pfem surface diffusion}. For solid-state dewetting problems, we always choose the contact line mobility $\eta=100$.

\subsection{Results for closed curves}

\begin{figure}[t!]
\centering
\includegraphics[width=0.5\textwidth]{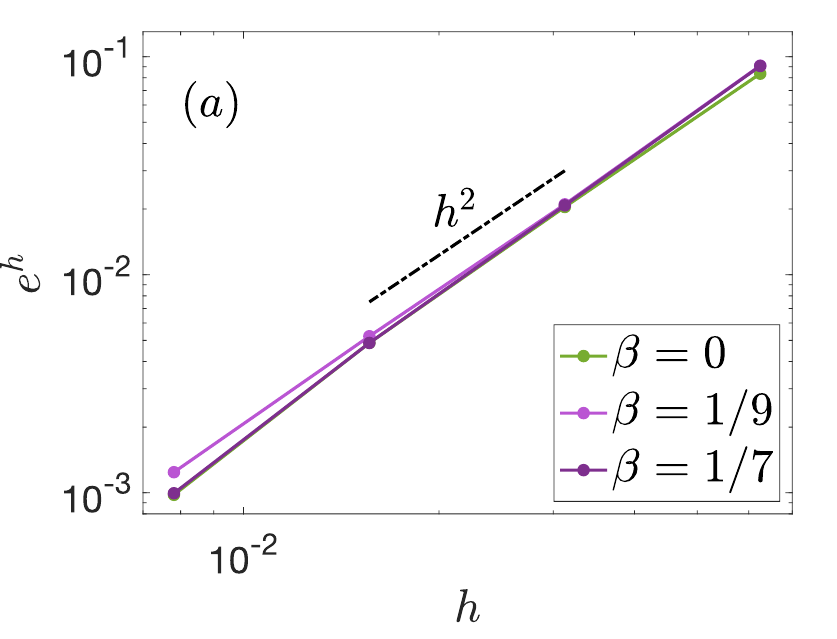}\includegraphics[width=0.5\textwidth]{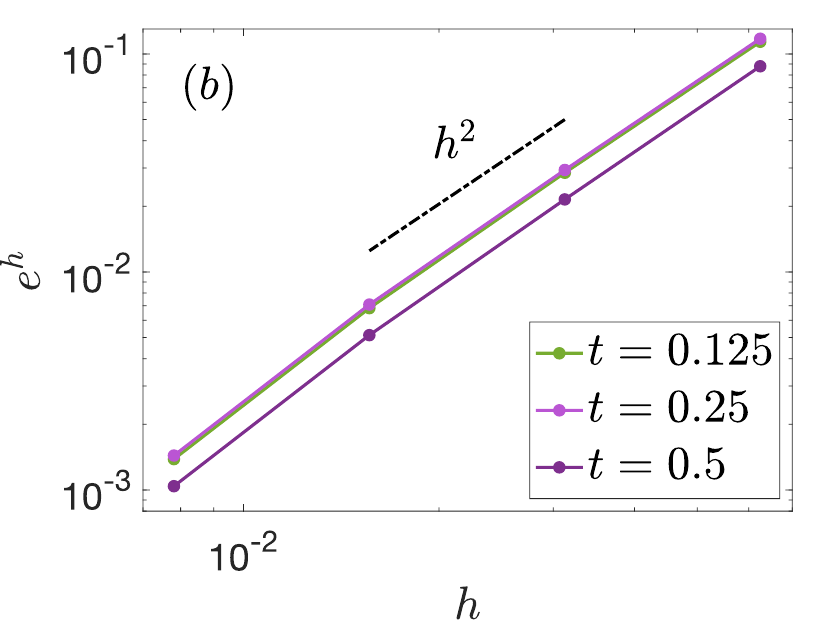}
\caption{Convergence rates of the SPFEM \eqref{eqn:sp-pfem surface diffusion} with $k(\theta)=k_0(\theta)$ for: (a) anisotropy in Case I at $t=0.5$ with different $\beta$; and (b) anisotropy in Case II with $b=-0.8$ at different times $t=0.125,0.25,0.5$.}
\label{fig: convergent}
\end{figure}

Fig.~\ref{fig: convergent} plots the convergence rates of the proposed SPFEM \eqref{eqn:sp-pfem surface diffusion} for: (a) the $3$-fold anisotropy $\hat{\gamma}(\theta)=1+\beta\cos 3\theta$ with different anisotropic strengths $\beta$ under a fixed time $t=0.5$; (b) the ellipsoidal anisotropy $\hat{\gamma}(\theta)=\sqrt{1-0.8\cos^2\theta}$ at different times. It clearly demonstrates that the second-order spatial convergence remains consistent regardless of anisotropies and computational times, suggesting a high level of robustness in the convergence rate.

\begin{figure}[t!]
\centering
\includegraphics[width=0.5\textwidth]{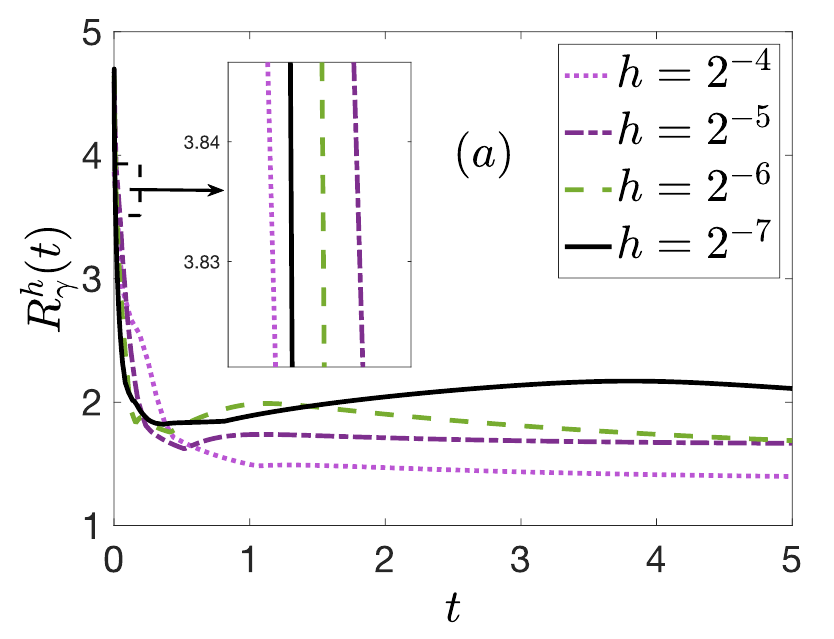}\includegraphics[width=0.5\textwidth]{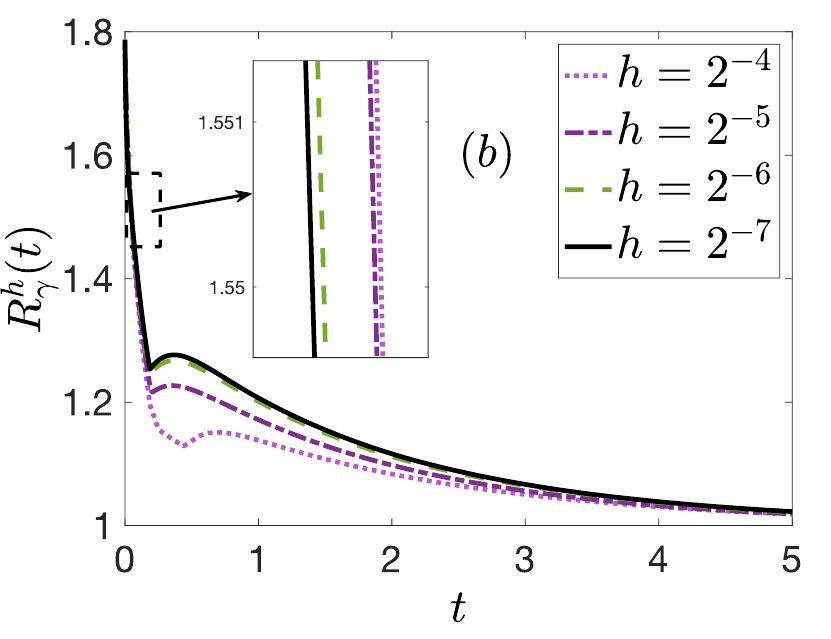}
\caption{Weighted mesh ratio of the SPFEM \eqref{eqn:sp-pfem surface diffusion} with $k(\theta)=k_0(\theta)$ for: (a) anisotropy in Case I with $\beta=\frac{1}{9}$; and (b) anisotropy in Case II with $b=-0.8$.}
\label{fig: mesh}
\end{figure}

Fig.~\ref{fig: mesh} exhibits that the weighted mesh ratio $R_{\gamma}^h$ converge to constants as $t\to+\infty$. This suggests an asymptotic quasi-uniform mesh distribution of the proposed SPFEM \eqref{eqn:sp-pfem surface diffusion}.

The time evolutions of the normalized area loss $\frac{\Delta A^h_c(t)}{A^h_c(0)}$, the number of the Newton's iteration with $h=2^{-7},\tau=2^{-10}$ are given in Fig.~\ref{fig: volume}. And the normalized energy $\frac{W^h_c(t)}{W^h_c(t)}$ with different $h$ are summarized in Fig.~\ref{fig: energy}.

\begin{figure}[t!]
\centering
\includegraphics[width=0.5\textwidth]{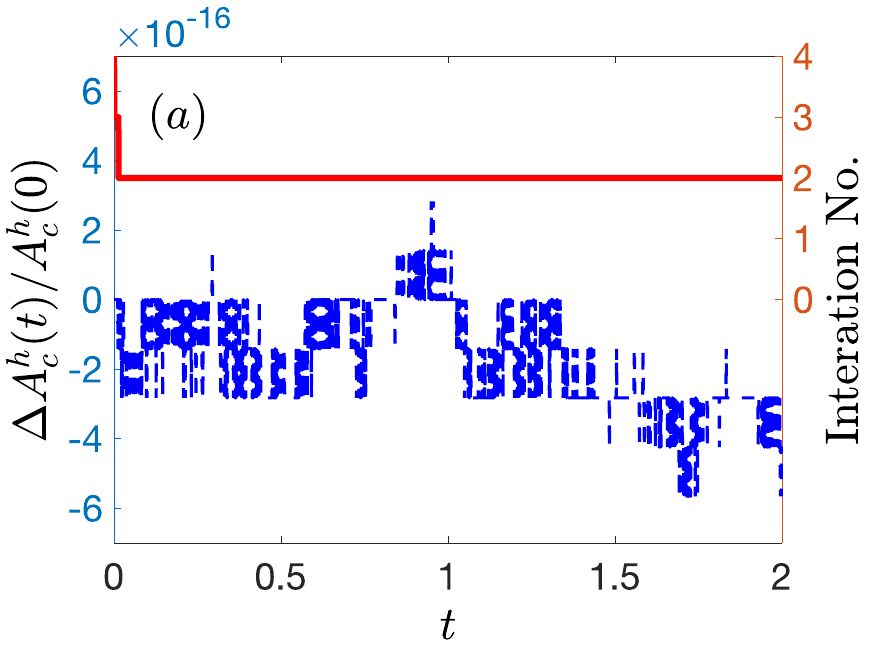}\includegraphics[width=0.5\textwidth]{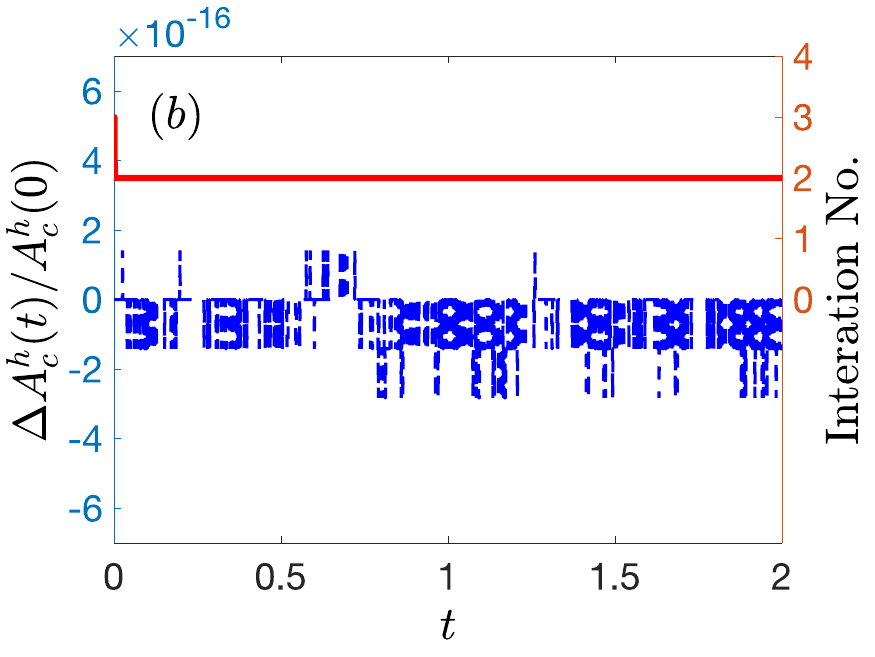}
\caption{Normalized area loss (blue dashed line) and iteration number (red line) of the SPFEM \eqref{eqn:sp-pfem surface diffusion} with $k(\theta)=k_0(\theta)$ and $h=2^{-7},\tau=2^{-10}$ for: (a) anisotropy in Case I with $\beta=\frac{1}{2}$; and (b) anisotropy in Case II with $b=-0.8$.}
\label{fig: volume}
\end{figure}

\begin{figure}[t!]
\centering
\includegraphics[width=0.5\textwidth]{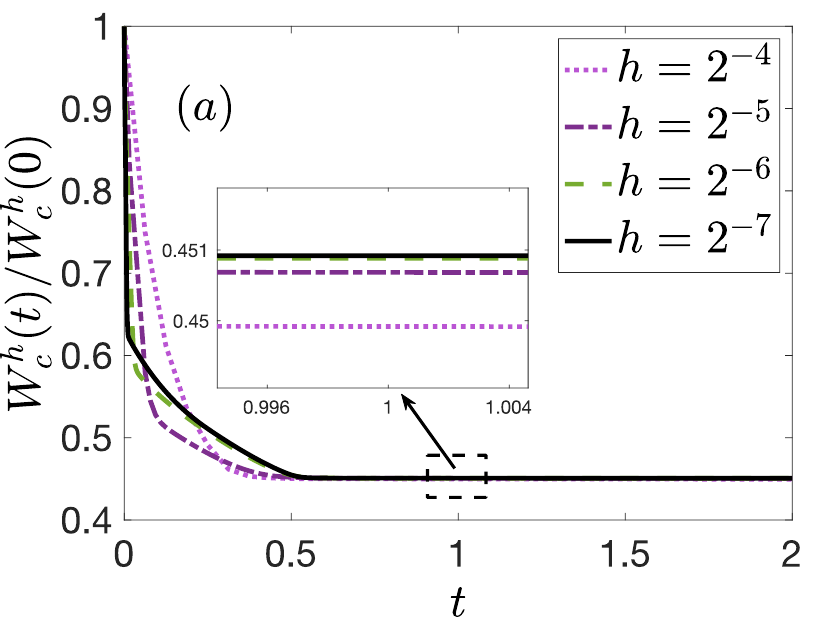}\includegraphics[width=0.5\textwidth]{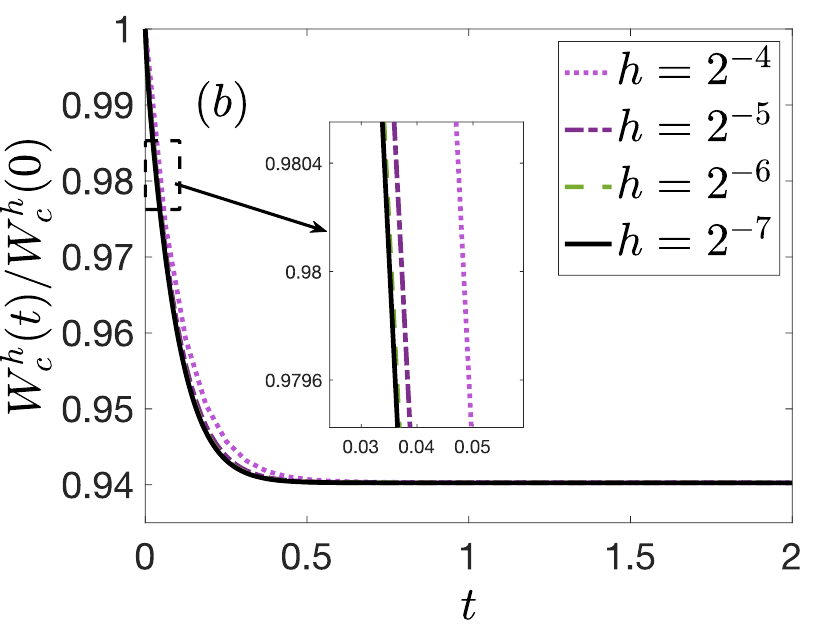}
\caption{Normalized energy of the SPFEM \eqref{eqn:sp-pfem surface diffusion} with $k(\theta)=k_0(\theta)$ for: (a) anisotropy in Case I with $\beta=\frac{1}{2}$; and (b) anisotropy in Case II with $b=-0.8$.}
\label{fig: energy}
\end{figure}

The observation from Fig.~\ref{fig: volume}--Fig.~\ref{fig: energy} reveals that: \begin{enumerate}
    \item The normalized area loss is at $10^{-16}$, aligns closely with the order of the round-off error (cf. Fig.~\ref{fig: volume}). This observation affirms the practical preservation of area in simulations.\\
    \item The numbers of the Newton's iteration are initially $3$ or $4$, and quickly descend to $2$ (cf. Fig.~\ref{fig: volume}). This discovery indicates that the proposed SPFEM \eqref{eqn:sp-pfem surface diffusion} can be solved with high efficiency, requiring only a few iterations.\\
    \item The normalized energy is monotonically decreasing when $\hat{\gamma}(\theta)$ satisfies the energy stable conditions \eqref{es condition on gamma} in Definition~\ref{def anisotropic stable} (cf. Fig.~\ref{fig: energy}). Results in Fig.~\ref{fig: energy} (a) shows that the proposed SPFEM \eqref{eqn:sp-pfem surface diffusion} still preserves good energy stability properties when $\beta$ takes its maximum value of $1/2$ in Remark~\ref{kfold remark}. And results in Fig.~\ref{fig: energy} (b) indicate that, unlike the ES-PFEM in \cite{li2021energy}, the SPFEM remains unconditionally energy stable when $b<-1/2$ as stated in Remark~\ref{ellipsoidal remark}.
\end{enumerate}

\subsection{Results for open curves in solid-state dewetting}
\begin{figure}[t!]
\centering
\includegraphics[width=0.5\textwidth]{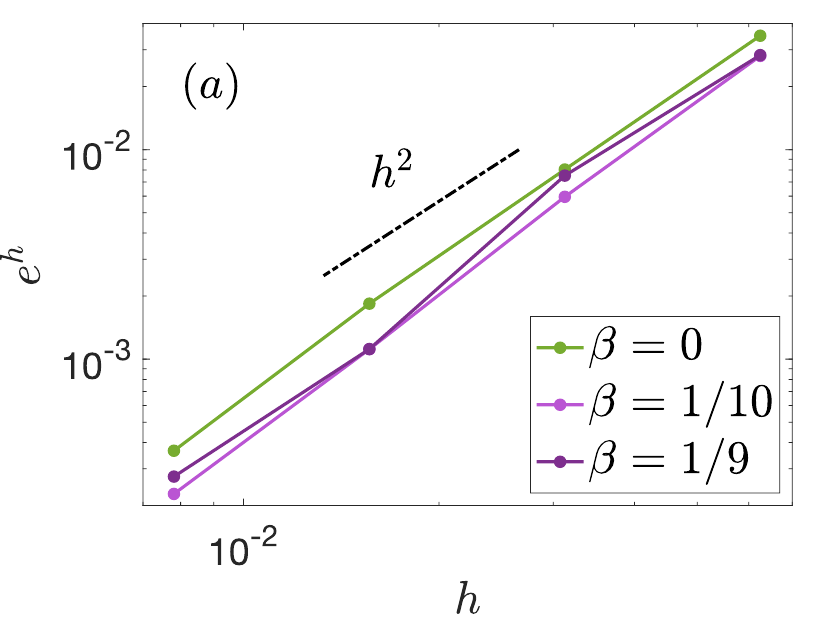}\includegraphics[width=0.5\textwidth]{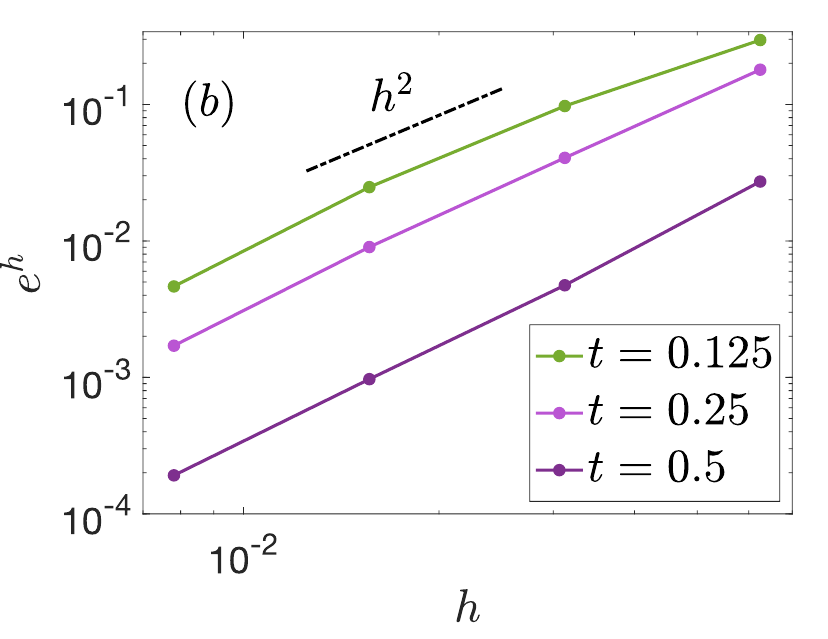}
\caption{Convergence rates of the SPFEM \eqref{eqn:sp-pfem surface diffusion} with $k(\theta)=k_0(\theta)$ and $\sigma=-\frac{\sqrt{2}}{2}$ for: (a) anisotropy in Case I at $t=0.5$ with different $\beta$; and (b) anisotropy in Case II with $b=2$ at different times $t=0.125,0.25,0.5$.}
\label{fig: convergent_open}
\end{figure}

Fig.~\ref{fig: convergent_open} plots the computation errors of the proposed SPFEM \eqref{eqn:sp-pfem solid-state dewetting} for: (a) the $3$-fold anisotropy $\hat{\gamma}(\theta)=1+\beta\cos 3\theta$ with different anisotropic strengths $\beta$ under a fixed time $t=0.5$; (b) the ellipsoidal anisotropy $\hat{\gamma}(\theta)=\sqrt{1+2\cos^2\theta}$ at different times. The results verify the quadratic convergence rate for the proposed SPFEM \eqref{eqn:sp-pfem solid-state dewetting}.

\begin{figure}[t!]
\centering
\includegraphics[width=0.5\textwidth]{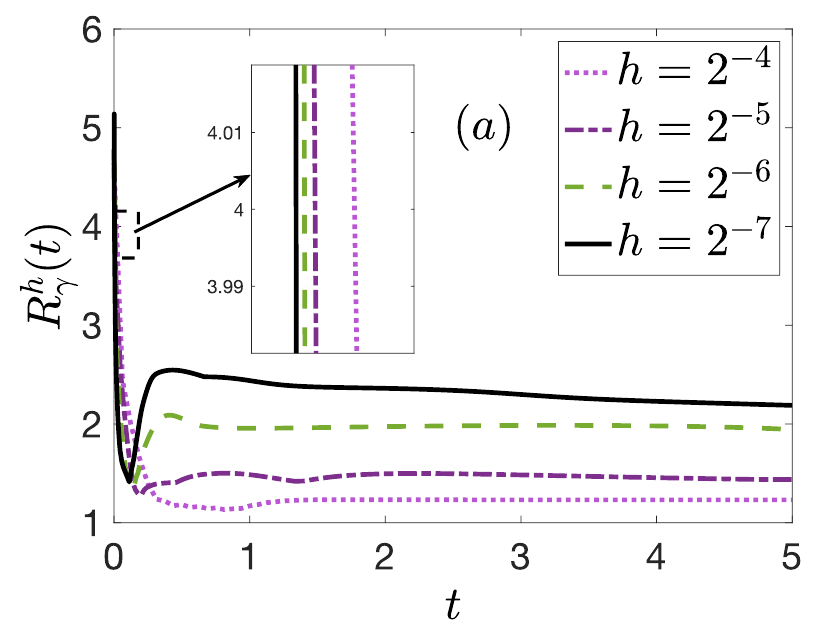}\includegraphics[width=0.5\textwidth]{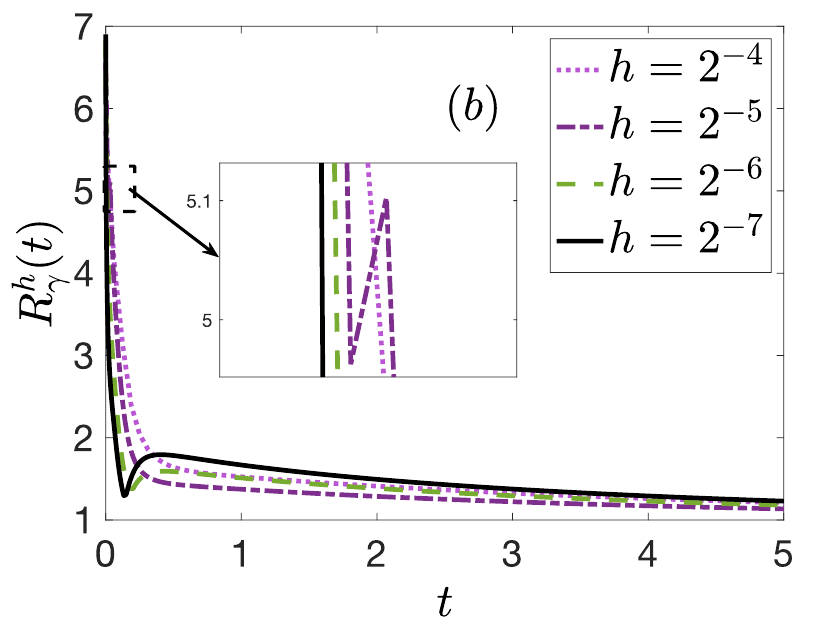}
\caption{Weighted mesh ratio of the SPFEM \eqref{eqn:sp-pfem surface diffusion} with $k(\theta)=k_0(\theta)$ and $\sigma=-\frac{\sqrt{2}}{2}$ for: (a) anisotropy in Case I with $\beta=\frac{1}{9}$; (b) anisotropy in Case II with $b=2$.}
\label{fig: mesh_open}
\end{figure}

In Fig.~\ref{fig: mesh_open}, the weighted mesh ratios $R_\gamma^h$ tend to constants as $t\to+\infty$, showing that the SPFEM \eqref{eqn:sp-pfem solid-state dewetting} still possesses the asymptotic quasi-uniform distribution.

Time evolutions of the normalized area loss $\frac{\Delta A^h_o(t)}{A^h_o(0)}$, the number of the Newton's iteration with $h=2^{-7},\tau=2^{-10}$ are presented in Fig.~\ref{fig: volume_open}. And the normalized energy $\frac{W^h_o(t)}{W^h_o(t)}$ with different $h$ are illustrated in Fig.~\ref{fig: energy_open}. 

\begin{figure}[t!]
\centering
\includegraphics[width=0.5\textwidth]{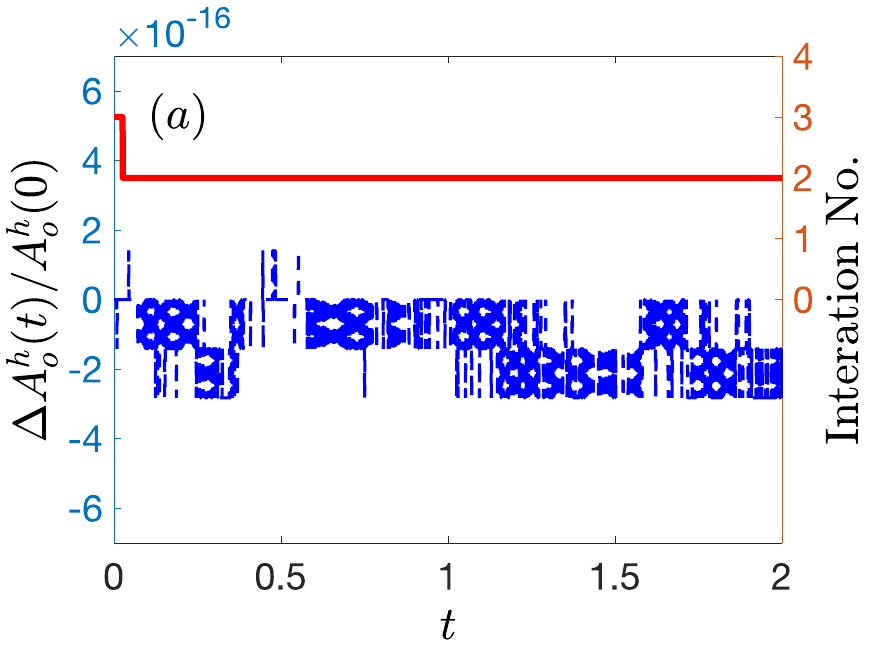}\includegraphics[width=0.5\textwidth]{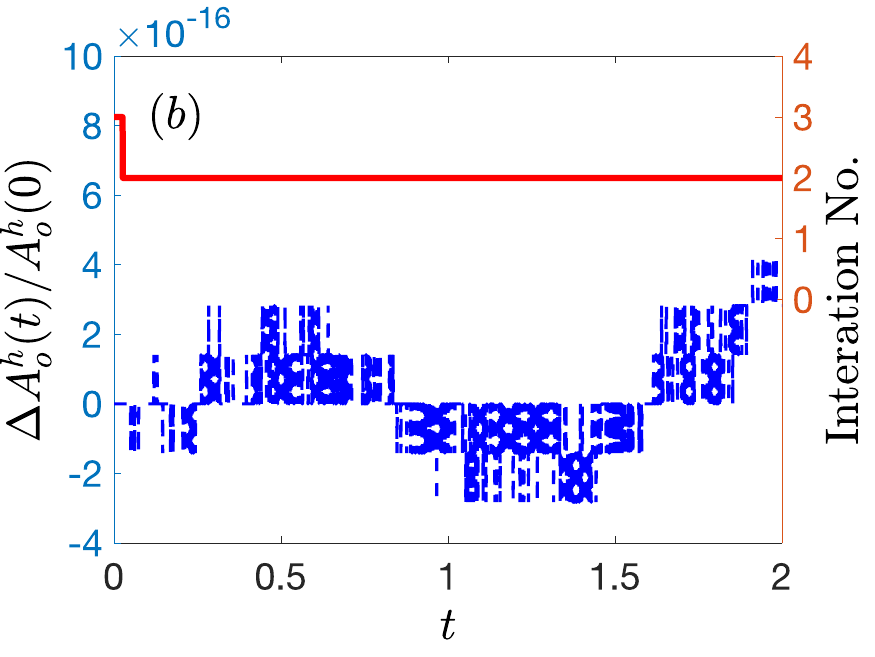}
\caption{Normalized area loss (blue dashed line) and iteration number (red line) of the SPFEM \eqref{eqn:sp-pfem surface diffusion} with $k(\theta)=k_0(\theta)$, $\sigma=-\frac{\sqrt{2}}{2}$ and $h=2^{-7},\tau=2^{-10}$ for: (a) anisotropy in Case I with $\beta=\frac{1}{9}$; and (b) anisotropy in Case II with $b=2$.}
\label{fig: volume_open}
\end{figure}

\begin{figure}[t!]
\centering
\includegraphics[width=0.5\textwidth]{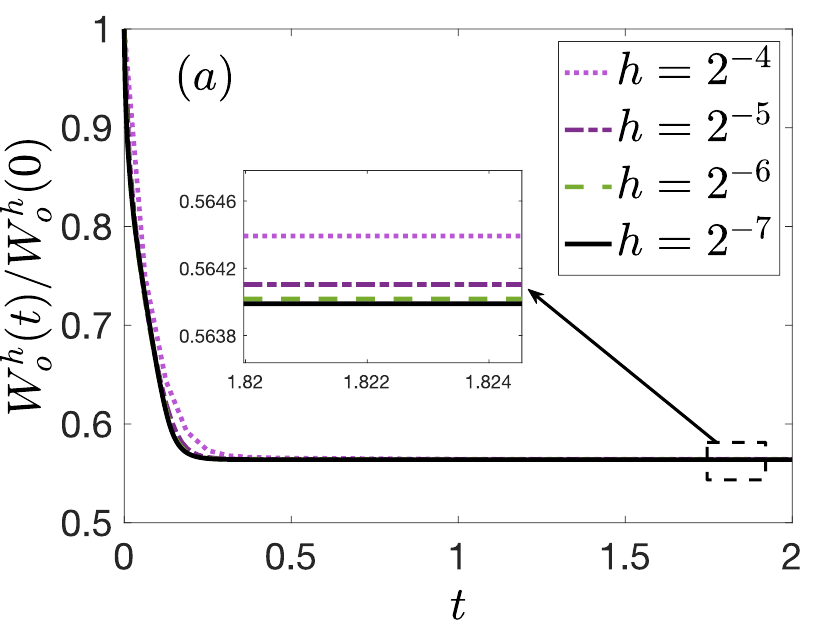}\includegraphics[width=0.5\textwidth]{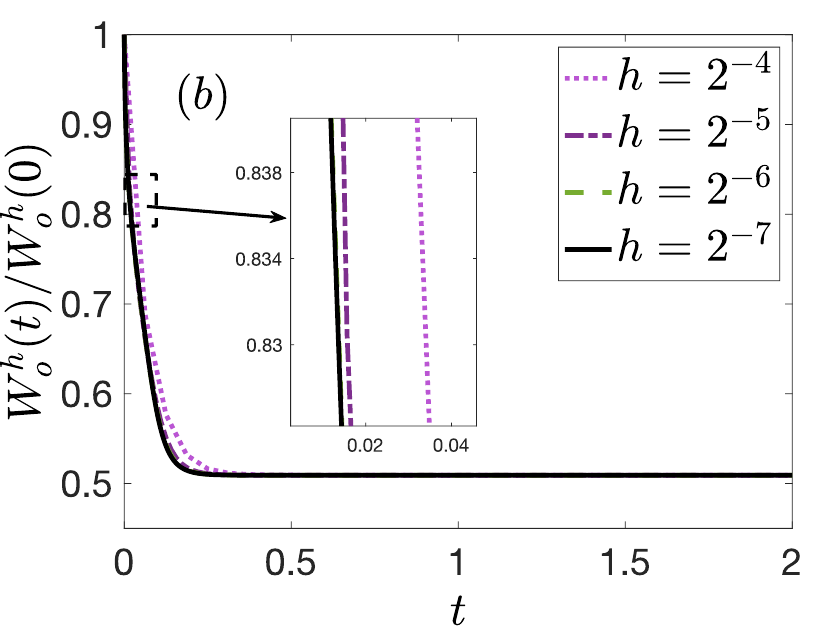}
\caption{Normalized energy of the SPFEM \eqref{eqn:sp-pfem surface diffusion} with $k(\theta)=k_0(\theta)$ and $\sigma=-\frac{\sqrt{2}}{2}$ for: (a) anisotropy in Case I with $\beta=\frac{1}{9}$; and (b) anisotropy in Case II with $b=2$.}
\label{fig: energy_open}
\end{figure}

It can be observed from Fig.~\ref{fig: energy}--Fig.~\ref{fig: volume_open} that: \begin{enumerate}
    \item The normalized area loss is about $10^{-16}$ at the same order of the round-off error (cf. Fig.~\ref{fig: volume_open}), verifying that the area is conserved up to the machine precision.\\
    \item The numbers of the Newton's iteration are initially $3$ and finally $2$ (cf. Fig.~\ref{fig: volume_open}). This finding suggests that, despite the fully-implicit nature of the proposed SPFEM \eqref{eqn:sp-pfem solid-state dewetting}, it can be solved very efficiently.\\
    \item The normalized energy is monotonically decreasing when $\hat{\gamma}(\theta)$ satisfying \eqref{es condition on gamma} (cf. Fig.~\ref{fig: energy_open}). In contrast to the ES-PFEM in \cite{li2021energy}, the proposed SPFEM \eqref{eqn:sp-pfem solid-state dewetting} still guarantees the energy dissipation when $\beta<\frac{1}{10}$ in Case I and $b>1$ in Case II, as asserted by Remark~\ref{kfold remark} and Remark~\ref{ellipsoidal remark}.
\end{enumerate}

\subsection{Application for morphological evolutions}

Finally we apply the proposed SPFEMs \eqref{eqn:sp-pfem surface diffusion} and \eqref{eqn:sp-pfem solid-state dewetting} to simulate the morphological evolutions under the anisotropic surface diffusion. Results for both closed curves and open curves in solid-state dewetting problems are provided.

The morphological evolutions from the initial shapes to their numerical equilibriums are presented in Fig.~\ref{fig: morphevolve}--Fig.~\ref{fig: morphevolve_open}. For closed curve cases, the initial shape is an ellipse with major axis $4$ and minor axis $1$, while for open curve cases, it is an open $4\times 1$ rectangle. 

\begin{figure}[t!]
\centering
\includegraphics[width=1.0\textwidth]{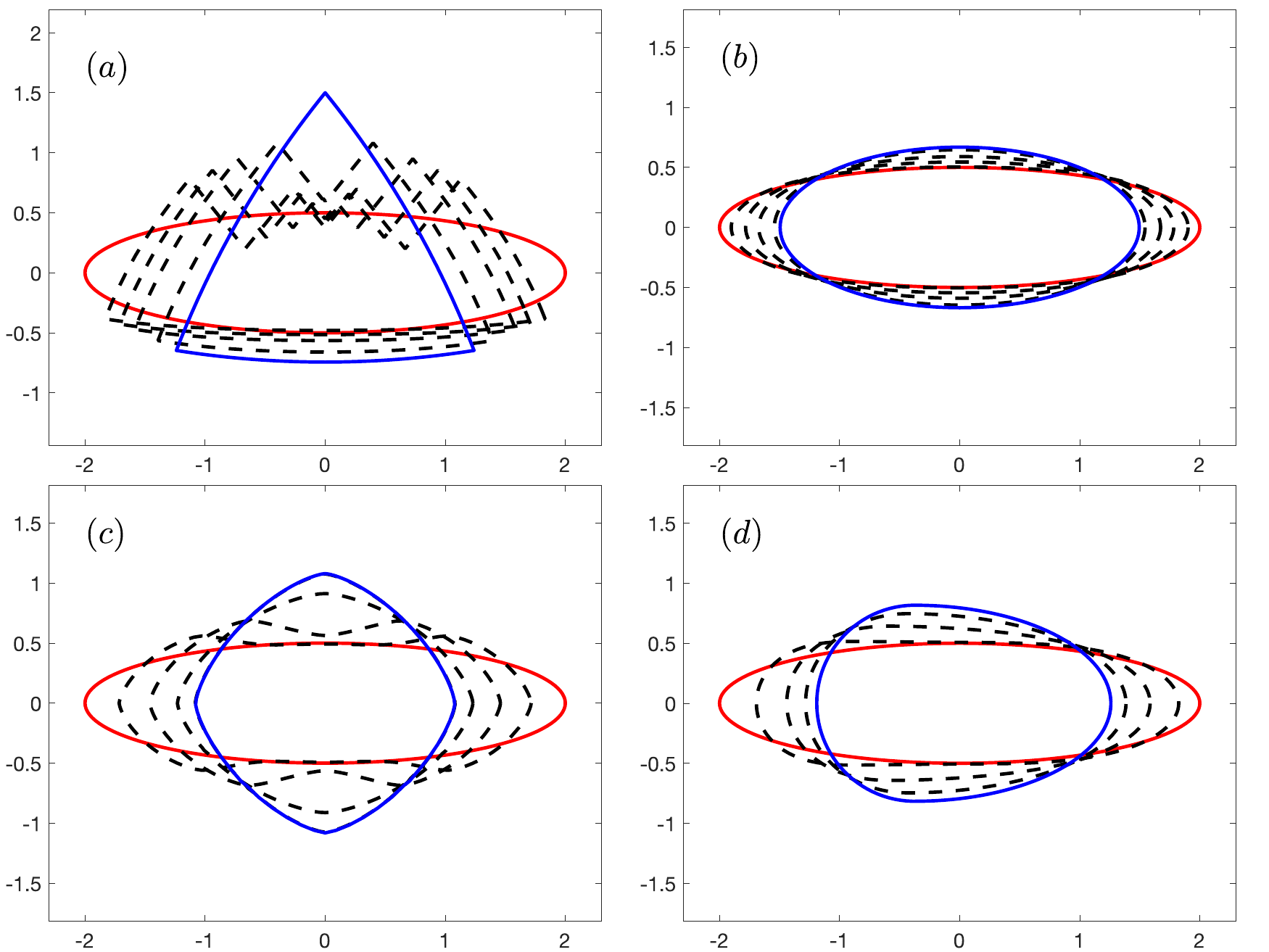}
\caption{Morphological evolutions of an ellipse with major axis $4$ and minor axis $1$ under anisotropic surface diffusion with different surface energies: (a) anisotropy in Case I with $\beta=\frac{1}{2}$; (b) anisotropy in Case II with $b=-0.8$; (c) $\hat{\gamma}(\theta)=1+\frac{1}{16}\cos 4\theta$; (d) $\hat{\gamma}(\theta)=\sqrt{\left(\frac{5}{2}+\frac{3}{2}\text{sgn}(n_1)\right)n_1^2+n_2^2}$ with $\boldsymbol{n}=(n_1,n_2)^T=(-\sin\theta,\cos\theta)^T$. The red and blue lines represent the initial shape and the numerical equilibrium, respectively; and the black dashed lines represent the intermediate curves. The mesh size and the time step are taken as $h=2^{-7},\tau=2^{-10}$.} 
\label{fig: morphevolve}
\end{figure}

Fig.~\ref{fig: morphevolve} plots the morphological evolutions of an ellipse with major axis $4$ and minor axis $1$ under anisotropic surface diffusion with four different surface energies: (a) anisotropy in Case I with $\beta=\frac{1}{2}$, which attends the maximum value in Remark~\ref{kfold remark}; (b) anisotropy in Case II with $b=-0.8$; (c) the $4$-fold anisotropy $\hat{\gamma}(\theta)=1+\frac{1}{16}\cos 4\theta$
\cite{bao2017parametric}; and (d) $\hat{\gamma}(\theta)=\sqrt{\left(\frac{5}{2}+\frac{3}{2}\text{sgn}(n_1)\right)n_1^2+n_2^2}$ with $\boldsymbol{n}=(n_1,n_2)^T=(-\sin\theta,\cos\theta)^T$ \cite{deckelnick2005computation}. 

Results in Fig.~\ref{fig: morphevolve} (b) and Fig.~\ref{fig: morphevolve} (c) show that, compared to the ES-PFEM in \cite{li2021energy}, the proposed SPFEM \eqref{eqn:sp-pfem surface diffusion} demonstrates a better performance over a broader range of parameters during evolutions. Fig.~\ref{fig: morphevolve} (d) indicates that the SPFEM \eqref{eqn:sp-pfem surface diffusion} also works well for a globally $C^1$ and piecewise $C^2$ anisotropy.

\begin{figure}[t!]
\centering
\includegraphics[width=1.0\textwidth]{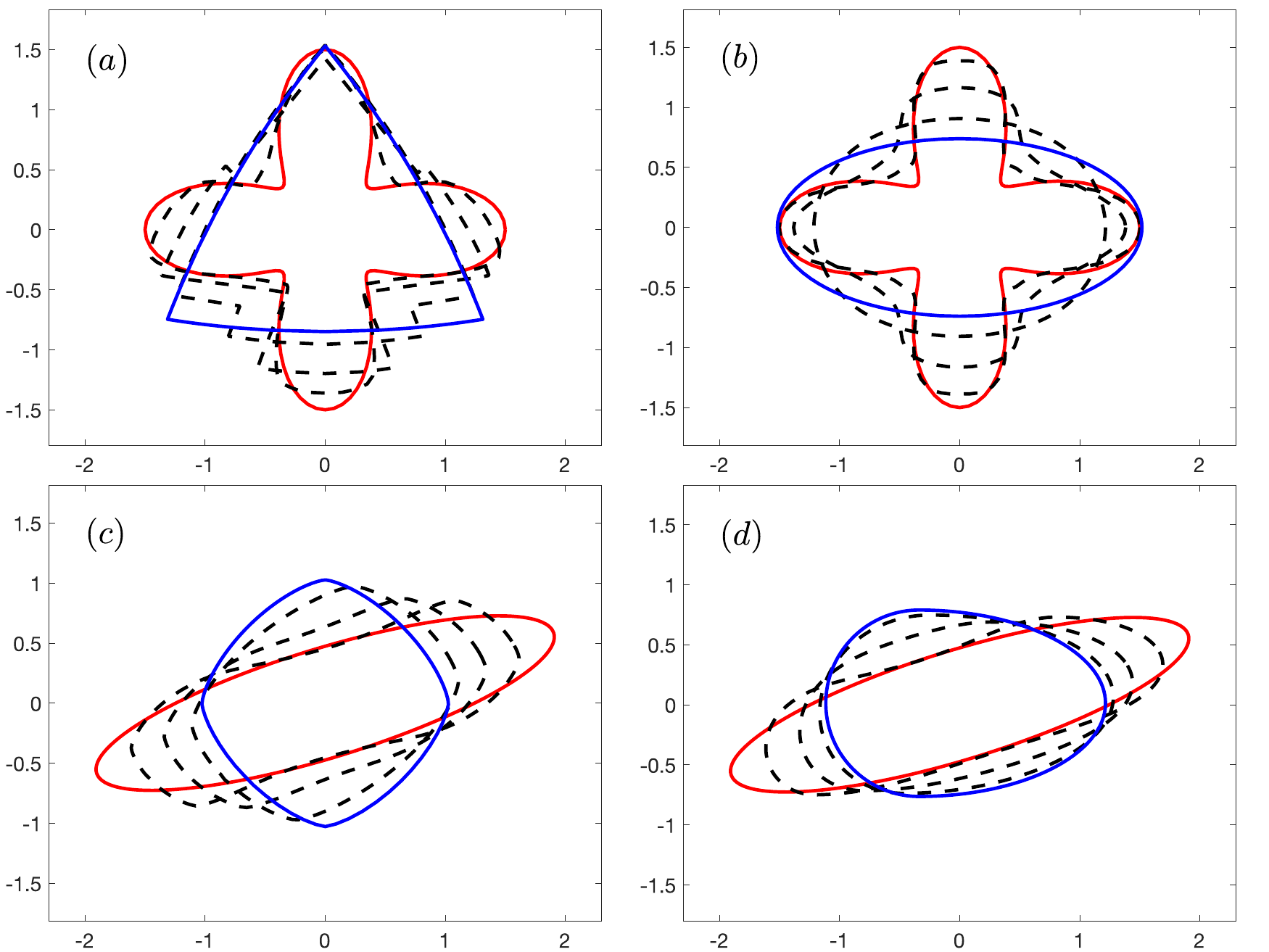}
\caption{Morphological evolutions of a four-fold star shape curve and a rotated ellipse under anisotropic surface diffusion with different surface energies: (a) anisotropy in Case I with $\beta=\frac{1}{2}$; (b) anisotropy in Case II with $b=-0.8$; (c) $\hat{\gamma}(\theta)=1+\frac{1}{16}\cos 4\theta$; (d) $\hat{\gamma}(\theta)=\sqrt{\left(\frac{5}{2}+\frac{3}{2}\text{sgn}(n_1)\right)n_1^2+n_2^2}$ with $\boldsymbol{n}=(n_1,n_2)^T=(-\sin\theta,\cos\theta)^T$. The red and blue lines represent the initial shape and the numerical equilibrium, respectively; and the black dashed lines represent the intermediate curves. The mesh size and the time step are taken as $h=2^{-7},\tau=2^{-10}$.} 
\label{fig: morphevolve2}
\end{figure}

And Fig.~\ref{fig: morphevolve2} illustrates the morphological evolutions under anisotropic surface diffusion with different surface energies for a four-fold star shape curve \begin{equation}
    \left\{\begin{array}{l}
        x=(1+0.5\cos 4\theta)\cos\theta, \\
        y=(1+0.5\cos 4\theta)\sin\theta,
    \end{array}\right.\theta\in 2\pi\mathbb{T},
\end{equation} and an ellipse (with major axis $4$ and minor axis $1$) rotated counterclockwise by $\frac{\pi}{10}$. 

\begin{figure}[t!]
\centering
\includegraphics[width=1.0\textwidth]{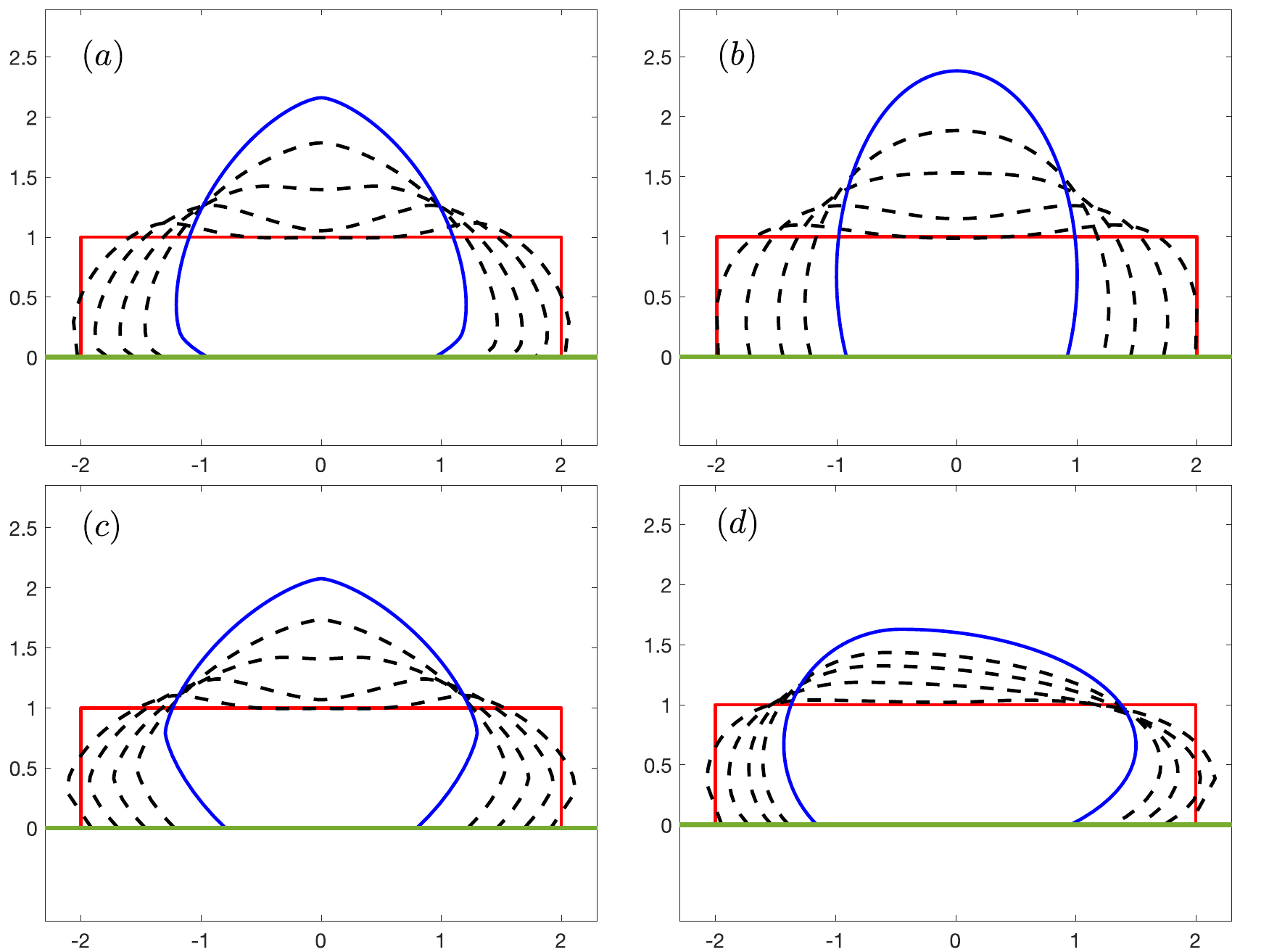}
\caption{Morphological evolutions of an open $4\times 1$ rectangle under anisotropic surface diffusion with different surface energies: (a) anisotropy in Case I with $\beta=\frac{1}{9}$; (b) anisotropy in Case II with $b=2$; (c) $\hat{\gamma}(\theta)=1+\frac{1}{16}\cos 4\theta$; (d) $\hat{\gamma}(\theta)=\sqrt{\left(\frac{5}{2}+\frac{3}{2}\text{sgn}(n_1)\right)n_1^2+n_2^2}$ with $\boldsymbol{n}=(n_1,n_2)^T=(-\sin\theta,\cos\theta)^T$. The red and blue lines represent the initial shape and the numerical equilibrium, respectively; and the black dashed lines represent the intermediate curves. The parameters are choosen as $\sigma=-\frac{\sqrt{2}}{2},h=2^{-7},\tau=2^{-10}$.} 
\label{fig: morphevolve_open}
\end{figure}

In Fig.~\ref{fig: morphevolve_open}, we display the morphological evolutions from an open $4\times 1$ rectangular curve to their equilibriums shapes with different surface energies: (a) anisotropy in Case I with $\beta=\frac{1}{9}$; (b) anisotropy in Case II with $b=2$; (c) the $4$-fold anisotropy $\hat{\gamma}(\theta)=1+\frac{1}{16}\cos 4\theta$; (d) $\hat{\gamma}(\theta)=\sqrt{\left(\frac{5}{2}+\frac{3}{2}\text{sgn}(n_1)\right)n_1^2+n_2^2}$ with $\boldsymbol{n}=(n_1,n_2)^T=(-\sin\theta,\cos\theta)^T$. 

Similar to the closed curve cases, the SPFEM \eqref{eqn:sp-pfem solid-state dewetting} extends the choices in surface energies for simulating solid-state dewetting (cf. Fig.~\ref{fig: morphevolve_open} (a) -- Fig.~\ref{fig: morphevolve_open} (c)). And Fig.~\ref{fig: morphevolve_open} (d) illustrates that our method also performs effectively for $\hat{\gamma}(\theta)$ with lower regularity.

\section{Conclusions}\label{section conclusions}

We propose a structure-preserving stabilized parametric finite element method (SPFEM) for the anisotropic surface diffusion. This method is subject to mild conditions on $\hat{\gamma}(\theta)$, and works effectively for closed curves and open curves with contact line migration in solid-state dewetting. By introducing a new stabilized surface energy matrix, we obtain a conservative form and its weak formulation for anisotropic surface diffusion. Based on this weak formulation, a novel SPFEM is proposed by utilizing the PFEM for spatial discretization and the implicit-explicit Euler method for temporal discretization. To analyze the unconditional energy stability, we extend the framework proposed by Bao and Li to the $\hat{\gamma}(\theta)$ formulation. This approach starts by defining the minimal stabilizing function, proving its existence, results in a local energy estimate, and subsequently establishes unconditional energy stability. Due to the very mild requirements on the surface energy $\hat{\gamma}(\theta)$, the methods are able to simulate over a broader range of anisotropies for both closed curves and open curves. Moreover, the SPFEMs are applicable for the globally $C^1$ and piecewise $C^2$ anisotropy as well, which is a capability not possessed by other PFEMs.


\section*{CRediT authorship contribution statement}

Yulin Zhang: Conceptualization, Methodology, Software, Investigation, Writing - Original Draft.\\

Yifei Li: Project administration, Supervision, Writing - Review \& Editing.\\

Wenjun Ying: Supervision, Writing - Review \& Editing.


\section*{Declaration of competing interest}

The authors declare that they have no known competing financial interests or personal relationships that could have appeared to influence the work reported in this paper.

\section*{Acknowledgement}

We would like to specially thank Professor Weizhu Bao for his valuable suggestions and comments. The work of Zhang was partially supported by the State Scholarship Fund (No.202306230346) by the Chinese Scholar Council (CSC) and the Zhiyuan Honors Program for Graduate Students of Shanghai Jiao Tong University (No.021071910051). The work of Li was funded by the Ministry of Education of Singapore under its AcRF Tier 2 funding MOE-T2EP20122-0002 (A-8000962-00-00). Part of the work was done when the authors were visiting the Institute of Mathematical Science at the National University of Singapore in 2024.

\section*{}

\end{document}